\documentclass[12pt]{amsart}
\usepackage{amsmath}
\usepackage{amssymb}
\usepackage{amsfonts}
\usepackage{amsthm}
\usepackage{mathrsfs}
\usepackage[all]{xy}
\usepackage[pdftex]{graphicx}
\usepackage{color}
\usepackage{cite}

\begin{document}

\theoremstyle{plain}
\newtheorem{thm}{Theorem}
\newtheorem{lemma}[thm]{Lemma}
\newtheorem{cor}[thm]{Corollary}
\newtheorem{conj}[thm]{Conjecture}
\newtheorem{prop}[thm]{Proposition}
\newtheorem{heur}[thm]{Heuristic}

\theoremstyle{definition}
\newtheorem{defn}[thm]{Definition}
\newtheorem*{ex}{Example}
\newtheorem*{notn}{Notation}

\title[Birth and Death Cutoff]{The Cutoff Phenomenon for Random Birth and Death Chains}
\author{Aaron Smith}

\address{Institute for Computational and Experimental Research in Mathematics,Brown University, Providence, RI}
\email{asmith3@math.stanford.edu}
\date{\today}
\maketitle

\section{Abstract}

For any distribution $\pi$ with support equal to $[n] = \{ 1, 2, \ldots, n \}$, we study the set $\mathcal{A}_{\pi}$ of tridiagonal stochastic matrices $K$ satisfying $\pi(i) K[i,j] = \pi(j) K[j,i]$ for all $i, j \in [n]$. These matrices correspond to birth and death chains with stationary distribution $\pi$. We study matrices $K$ drawn uniformly from $\mathcal{A}_{\pi}$, following the work of Diaconis and Wood on the case $\pi(i) = \frac{1}{n}$. We analyze a `block sampler' version of their algorithm for drawing from $\mathcal{A}_{\pi}$ at random, and use results from this analysis to draw conclusions about typical matrices. The main result is a soft argument comparing cutoff for sequences of random birth and death chains to cutoff for a special family of birth and death chains with the same stationary distributions. 

\section{Introduction}
In \cite{DiWo10}, Diaconis and Wood study the collection $\mathcal{A}_{(\frac{1}{n}, \frac{1}{n}, \ldots, \frac{1}{n})}$ of $n$ by $n$ doubly-stochastic tridiagonal matrices. These matrices are the transition kernels of birth and death (BD) chains with uniform stationary distribution, and \cite{DiWo10} uses detailed knowledge of the set $\mathcal{A}_{(\frac{1}{n}, \frac{1}{n}, \ldots, \frac{1}{n})}$ to obtain information about `typical' birth and death chain with uniform stationary distribution. In this paper, we extend many of their results to birth and death chains with general stationary distribution. \par 

There are two main contributions. Section \ref{SecGibbs1} contains an analysis of a `block sampler' version of their algorithm for sampling from the set of transition kernels with a given stationary distribution, and proves that the algorithm is quite efficient. Sections \ref{SecCutoffBasics} to \ref{SecNonCutoff} contain a description of the cutoff phenomenon for random sequences of birth and death chains with given stationary distributions. More precisely, Section \ref{SecCutoffBasics} introduces the results that will be used to describe cutoff, and section \ref{SecCutoffExamples} proves that some natural sequences of random birth and death chains do exhibit cutoff. Our primary result, the statement of which is slightly technical, is in Section \ref{SecNonCutoff}. This generalizes the main result of \cite{DiWo10} proving the lack of cutoff for random birth and death chains with uniform stationary distribution to many other families of birth and death chains using a probabilistic comparison technique.  \par 

We begin with some notation. Let $\pi$ be a distribution on $[n] = \{ 1, 2, \ldots, n \}$, and consider the set  $\mathcal{A}_{\pi}$ of tridiagonal stochastic matrices $K$ satisfying $\pi(i) K[i,j] = \pi(j) K[j,i]$ for all $i, j \in [n]$. This set may be viewed as a compact convex subset of $\mathbb{R}^{n-1}$, with the matrices parameterized by their super-diagonal entries $K[i,i+1]$. With the convention $c_{0} = c_{n} = 0$, we associate to a sequence $c_{i}$ satisfying 
\begin{equation} \label{BdIneqAllowableChains}
0 \leq c_{i} \leq \min \left( 1 - \frac{\pi(i-1)}{\pi(i)} c_{i-1} , \frac{\pi(i+1)}{\pi(i)} (1 -c_{i+1}) \right)
\end{equation}
a matrix $K \in \mathcal{A}_{\pi}$ given by:
\begin{align} \label{EqSupDiagRep}
K[i,i+1] &= c_{i} \\
K[i+1,i] &= \frac{\pi(i)}{\pi(i+1)}c_{i} \\
K[i,i] &= 1 - c_{i} - \frac{\pi(i-1)}{\pi(i)} c_{i-1} \\
K[i,j] &= 0, \, \, \, \, \, \, \vert i - j \vert > 1
\end{align}
It is easy to see that there is a similar parameterization by the sub-diagonal entries $K[i+1,i]$. We will sometimes need to use this other representation. \par 

We now recall some standard definitions from the theory of Markov chains. Let $\mu$, $\nu$ be two distributions on the measure space $(\Omega, \mathcal{A})$. The total variation distance between $\mu$ and $\nu$ is given by 
\begin{equation*}
\vert \vert \mu - \nu \vert \vert_{TV} = \sup_{A \in \mathcal{A}} \vert \mu(A) - \nu(A) \vert
\end{equation*}
Then the \textit{mixing profile} of an ergodic Markov chain $X_{t}$ with stationary distribution $\pi$ is given by
\begin{equation*}
\tau(\epsilon) = \sup_{X_{0} \in \Omega} \inf \{ t : \vert \vert \mathcal{L}(X_{t}) - \pi \vert \vert_{TV} < \epsilon \}
\end{equation*}
for any $0 < \epsilon < 1$. \par 
It is well-known that $d(t) \equiv \sup_{X_{0}, Y_{0} \in \Omega} \vert \vert \mathcal{L}(X_{t}) - \mathcal{L}(Y_{t}) \vert \vert_{TV}$ is submultiplicative and satisfies $d(t) \geq \vert \vert \mathcal{L}(X_{t}) - \pi \vert \vert_{TV}$. However, for many well-known sequences of Markov chains, the distance to stationarity drops from very close to 1 to very close to 0 over a distance that is much smaller than the $O(\tau(\frac{1}{4}))$ timescale suggested by submultiplicativity. This is known as the cutoff phenomenon. More precisely, a sequence $X_{t}^{(n)}$ of Markov chains with mixing profile $\tau_{n}(\epsilon)$ exhibits cutoff if
\begin{equation*}
\lim_{n \rightarrow \infty} \frac{\tau_{n}(\epsilon)}{\tau_{n}(1 - \epsilon)} = 1
\end{equation*}
for all $0 < \epsilon < 1$. \par 
The cutoff phenomenon gained its current name in \cite{AlDi86}. Over the following decade, there was a great deal of effort to prove that certain natural families of Markov chains exhibited cutoff (see \cite{Diac96} for a survey of early results). Almost all of these families consisted of chains with a great deal of symmetry. Recently, there has been a great deal of success in analyzing more complicated examples. Some of this progress, such as Sly and Lubetzky's amazing paper \cite{LuSl11}, consists of detailed analyses of specific chains. More importantly for this paper, general (and often easily checkable) necessary or sufficient conditions for cutoff have been found \cite{DiSa08} \cite{ChSa08} \cite{DLP08}. \par 
This paper will rely on a characterization of Total Variation cutoff for birth and death chains found in \cite{DLP08}. The necessary condition for cutoff, which applies for all sequences of Markov chains, is as follows. Let $K_{n}$ be a sequence of kernels of Markov chains satisfying $K_{n}[i,i] \geq \frac{1}{2}$ for all $1 \leq i \leq n$; these are called $\frac{1}{2}$-lazy chains. Then $K_{n}$ is a symmetric matrix with all real and nonnegative eigenvalues, the largest of which is 1. Let $\lambda_{n}$ be the second largest eigenvalue of $K_{n}$. Then let $\tau_{n} = \tau_{n}(\frac{1}{4})$ be its mixing time, and call $\frac{1}{1 - \lambda_{n}}$ its \textit{relaxation time}. Then if $\lim_{n \rightarrow \infty} \tau_{n} (1 - \lambda_{n}) < \infty$, the sequence of Markov chains doesn't exhibit cutoff (see proposition 18.4 of \cite{LPW09}). The much more difficult partial converse, Theorem 1 of \cite{DLP08}, is:

\begin{thm}[Cutoff for BD Chains] \label{BdThmCutoffCrit}
For any fixed $0 < \epsilon < \frac{1}{2}$, there exists a constant $C(\epsilon)$ so that for any $\frac{1}{2}$-lazy BD chain $X_{t}$ with spectral gap $\gamma$,
\begin{equation*}
\tau(\epsilon) - \tau(1- \epsilon) < C(\epsilon) \sqrt{\frac{\tau(\frac{1}{4})}{\gamma} }
\end{equation*}
\end{thm}

The work on cutoff in this paper was inspired by the question of whether or not cutoff is `generic'. That is, do typical sequences of Markov chains exhibit cutoff? There has been a great deal of work on this question for random walks on groups, going back to work of Dou, Hildebrand and Wilson \cite{DoHi95} \cite{Wils97} and described in the survey paper \cite{Hild05}. More recently, this has progressed to other random walks with uniform stationary distribution \cite{LuSl09} \cite{DiWo10}. To our knowledge, this paper contains the first results describing cutoff for families of random chains with non-uniform stationary distribution. Looking at non-uniform distributions poses new challenges, and our techniques rely on comparing these non-uniform chains to related uniform chains. \par 

Throughout this paper, the kernel $K$ is always chosen uniformly from the set $\mathcal{A}_{\pi}$, but cutoff is always discussed for the associated $\frac{1}{2}$-lazy kernel, $K' = \frac{1}{2} (I + K)$. This allows Theorem \ref{BdThmCutoffCrit} to be applied. It is worth mentioning, in light of Theorem 3.1 of \cite{DLP08}, that there is nothing special about making the chain exactly $\frac{1}{2}$-lazy. In particular, for any fixed $0 < \delta < 1$, all theorems about the existence of cutoff for sequences of random chains $K' = \frac{1}{2} (I + K)$ will also apply to sequences of random chains $K'' = \delta I + (1 - \delta)K$, with modifications only to the cutoff location and bounds on the window size. Similar arguments can be made to apply for the original chain $K$ in many cases, but this requires substantial extra calculation. \\

\section{A Gibbs Sampler on $\mathcal{A}_{\pi}$} \label{SecGibbs1}
In Section 2.2 of \cite{DiWo10}, Diaconis and Wood propose several ways to sample from the set $\mathcal{A}_{\pi}$. They include an exact algorithm, which seemed to be slow in practice, and a Markov chain based algorithm without any rigorous running time bounds, which seemed to be quick in practice. This section describes an algorithm closely related to that used in \cite{DiWo10}, and provides a proof of its efficiency. Although the algorithm is quite efficient, it is not trivial to implement well. Its analysis is included here to answer a question posed in \cite{DiWo10}, and also because many of the lemmas proved in its analysis are useful for proving the main theorems about cutoff in this paper. For a practical discussion about implementing this sampler, a very different analysis resulting in a novel perfect sampling algorithm, and an $O(n \log(n)^{2})$ bound on the running time of the original algorithm, see \cite{Smit12a}.  \par

Define the space 
\begin{equation}
\mathcal{B}_{\pi} = \left\{ X \in \mathbb{R}^{n-1} \, : \, 0 \leq X[i] \leq \min \left( 1 - \frac{\pi(i-1)}{\pi(i)} X[i-1] , \frac{\pi(i+1)}{\pi(i)} (1 -X[i+1]) \right) \right\}
\end{equation}
 with the convention $X[0] = X[n] = 0$. From equation \eqref{BdIneqAllowableChains}, this is a parameterization of $\mathcal{A}_{\pi}$. Furthermore sampling uniformly from $\mathcal{B}_{\pi}$ is equivalent to sampling uniformly from $\mathcal{A}_{\pi}$. Next, fix an integer $k > 0$ and weight $w > 0$. \par 
We will define the following `block' Gibbs sampler $X_{t}$ on $\mathcal{B}_{\pi}$, for all $n > k$. At each step $t$, choose a coordinate $1 \leq i(t) \leq n - k $, according to the measure $P[i(t) = 1] = P[i(t) = n-k] = \frac{w}{n-k-2 + 2 w}$, and $P[i(t) = j] = \frac{1}{n-k-2 + 2 w}$ for $j \neq 1, n-k$. Next, update the entries of $X_{t}$ as follows. For $j \notin \{i(t), i(t) + 1, \ldots, i(t) + k -1 \}$, set $X_{t+1}[j] = X_{t}[j]$. For the other entries, choose $U_{t}$ uniformly in $\mathcal{B}_{\pi}$ conditioned on $U_{t}[j] = X_{t}[j]$ for all $j \notin \{i(t), i(t) + 1, \ldots, i(t) + k -1 \}$. Then, for all $j \in \{i(t), i(t) + 1, \ldots, i(t) + k -1 \}$, set $X_{t+1}[j] = U_{t}[j]$. The Gibbs sampler originally proposed in \cite{DiWo10} corresponded to the choice $w = k = 1$. \par 
Let $\mathcal{L}(X)$ denote the distribution of the random variable $X$, and let $U_{\pi}$ denote the uniform distribution on $\mathcal{B}_{\pi}$. The main result of this section is:

\begin{thm} [Convergence of Block Dynamics] \label{BdThmGenBlockDyn}
For $c>0$, $k = w = 55$, $t > 2 c n \log(n)$, and $n > 58$,
\begin{equation*}
\vert \vert \mathcal{L}(X_{t}) - U_{\pi} \vert \vert_{TV} \leq n^{1-c}
\end{equation*}
Conversely, for any choices of $k$ and $w$, any choice of $c > 0$, and $t = (\frac{1}{k} - c) n \log(n)$,
\begin{equation*}
\liminf_{n \rightarrow \infty} \vert \vert \mathcal{L}(X_{t}) - U_{\pi} \vert \vert_{TV} >0
\end{equation*}
\end{thm}

The proof proceeds by a path-coupling argument. There are many variants of this argument, which was introduced in \cite{BuDe97}. In this proof, 
Theorem 19 of \cite{Olli09} will be used. For a fixed transition kernel $K$ on a metric space $(\mathcal{X},d)$ of finite diameter $diam(\mathcal{X})$, the theorem may be stated as follows:

\begin{thm} [Path Coupling] \label{BdThmPathCoup}
Assume the metric space $(\mathcal{X},d)$ has the property that for any $x,y \in \mathcal{X}$, there exists a sequence of points $x = x_{0}, x_{1}, \ldots, x_{k} = y$ with the properties $d(x_{i}, x_{i+1}) \leq 1$ and $d(x,y) = \sum_{i = 0}^{k-1} d(x_{i}, x_{i+1})$. Assume that for all pairs of points $x,y$ with $d(x,y) \leq 1$, it is possible to couple a pair of Markov chains $(X_{t},Y_{t})$ started at $(X_{0}, Y_{0}) = (x,y)$, both with transition kernel $K$, so that $E[d(X_{1}, Y_{1})] \leq \alpha d(x,y)$ for some $\alpha < 1$. Then, if $X_{t}$ and $Y_{t}$ are two Markov chains with transition kernel $K$ and any initial distribution on $X$, they may be coupled so that for all $t \geq 0$,
\begin{equation*}
E[d(X_{t}, Y_{t})] \leq \alpha^{t} \sup_{p,q \in \mathcal{X}} d(p,q)
\end{equation*}
\end{thm}

To use the theorem, it is necessary to first define a metric and then prove a one-step contraction estimate for nearby points. Define the following Hamming-like metric on $\mathcal{B}_{\pi}$: $d(X,Y) = \sum_{j=1}^{n-1} \textbf{1}_{X[j] \neq Y[j]}$. The goal is to find a 1-step coupling for the block dynamics so that if $d(X_{0}, Y_{0}) = 1$, $E[d(X_{1},Y_{1})] \leq \alpha < 1$ under the coupling. Throughout this proof, all couplings will proceed by choosing the same block at each step for both chains. The simple strategy is  informed by the heuristic from the theory of spin systems that strong spatial mixing (of the underlying object being sampled) implies rapid temporal mixing (of a Gibbs sampler on that object); see \cite{DSVW04} for one explanation of the heuristic. In this case, the underlying object is the vector of super-diagonal entries with geometry being that of the path, and of course the Gibbs sampler is the block dynamics.  \par 

To turn the weak convergence bound of Theorem \ref{BdThmPathCoup} into a bound on Total Variation distance, we use the following following standard lemma (see Theorem 5.2 of \cite{LPW09}): 
\par
\begin{lemma} [Fundamental Coupling Lemma] \label{LemmaFundCoup}

 Assume $(X_{t}, Y_{t})$ is a coupling of Markov chains such that if $X_{s} = Y_{s}$, then $X_{t} = Y_{t}$ for all $t > s$. Assume also that $X_{0} = x$ and $Y_{0}$ is distributed according to the stationary distribution of $K$. Define the random time $\tau$ to be the first time at which $X_{t} = Y_{t}$. Then $\sup_{A \in \Sigma} \vert K^{t}(x,A) - \pi(A) \vert \leq P[\tau > t]$
\end{lemma}
\par

Assume for now that $X_{0}$ and $Y_{0}$ differ only at a single coordinate $j$ satisfying $n-k > j > k + 1$. Smaller and larger values of $j$ will be evaluated later. The goal now is to find a coupling which minimizes the expected distance after a single block has been updated. $j$ is inside the block with probability $\frac{k}{n -k - 2 + 2w}$. In this case, the two blocks are updated with exactly the same distribution, and so under the obvious coupling the distance decreases by 1. With probability $\frac{2}{n -k - 2 + 2w}$, the block ends at $j-1$ or begins at $j+1$. In this case, the distance generally increases, potentially by as much as $k$. The goal is to find a coupling under which the distance is unlikely to increase by a large number. The main step is the following lemma, which describes the spatial mixing of coordinates for elements chosen from the uniform distribution on $\mathcal{B}_{\pi}$: \\

\begin{lemma} [Super-Diagonal Mixing] \label{BdLmGeneralSdMixing}
Fix $1 \leq i < j \leq n-1$ and $0 \leq b,c \leq 1$. Let $Z$ and $Q$ be chosen from $\mathcal{B}_{\pi}$ chosen uniformly conditioned on $Z[i] = b_{1}$, $Q[i] = b_{2}$ and $Z[j] = Q[j] = a$. Then, for $2 \ell \leq j-2$,
\begin{equation*}
\vert \vert \mathcal{L}( Z[i+2\ell]) - \mathcal{L}( Q[i+2 \ell]) \vert \vert_{TV} \leq \prod_{q =1}^{\ell} (1 - R(i + 2 q))
\end{equation*}
where $R(q) = 1 - [C(q)-1] \log \left(\frac{1}{1 - \frac{1}{C(q)}}\right)$ and $C(q) = 16 \min \left( \frac{\pi(q+3)}{\pi(q+2)}, 1 \right)$.
\end{lemma}

The proof of this lemma will occupy most of the remainder of this section. The first step is the following easy coupling lemma. Let $F$ be a differentiable cumulative distribution function on $[0,1]$ which satisfies
\begin{itemize}
\item $F(x) \geq x$
\item $F(x) \leq Cx$ for some $C > 1$ 
\item $F(x)$ is concave
\end{itemize} 
Then let $\Psi$ be the cumulative distribution function on $[0, C^{-1}]$ with density $\xi(x) = Z^{-1} \frac{1 - Cx}{1 - x}$, where $Z = 1 - [C-1] \log \left(\frac{1}{1 - C^{-1}}\right)$ is a normalizing constant. Then
\begin{lemma}[Minorization for Concave Lipschitz Distributions] \label{BdLemmaConcLipMin}
For $F$ and $\Psi$ described above, and $[a,b] \subset [0, C^{-1}]$,
\begin{equation*}
F(b)- F(a) \geq Z(\Psi(b) - \Psi(a))
\end{equation*}
\end{lemma} 
To prove this, note that
\begin{align*}
1 &= \int_{0}^{a} F'(x) dx + \int_{a}^{1} F'(x) dx \\
& \leq a F'(0) + (1-a)F'(a) \\
&\leq aC + (1-a)F'(a)
\end{align*}
where the first inequality is due to assumption $iii)$, and the second is due to assumption $ii)$. Thus,
\begin{equation*}
F'(x) \geq \frac{1-Cx}{1-x}
\end{equation*}
proving the lemma. $\square$ \par 

In order to use this, it is necessary to relate this coupling of a single entry in the updated block to some sort of coupling of the entire updated block. The following generalization of Theorem 4.1 of \cite{DiWo10} provides a strong way to do so:

\begin{thm} [Markov Property for Super-Diagonal Entries] \label{BdThmMarkovProp}
 Let $B$ be chosen uniformly from $\mathcal{B}_{\pi}$. Then for any $1 \leq i \leq n-2$, any real constants $a_{1}, \ldots, a_{i}$ in the interval $[0,1]$ and any $t \in [0,1]$, 
\begin{equation} \label{BdEqMarkovChain}
P[B[i+1] \leq t \vert B[1] = a_{1}, \ldots,B[i] = a_{i}] = P[B[i+1] \leq t \vert  B[i] = a_{i}]
\end{equation} 
\end{thm}
The proof is the same as in the reference, and both the idea and notation are copied. Define $X_{i} = B[i]$, and let
\begin{equation*}
f_{i,n}(x) = \int_{X_{i+1} = 0}^{\min(x, \frac{\pi(i+2)}{\pi(i+1)})} \int_{X_{i+3}=0}^{\min(1 - \frac{\pi(i+2)}{\pi(i+3)} X_{i+2}, \frac{\pi(i+4)}{\pi(i+3)})} \ldots \int_{X_{n-1} = 0}^{\min(1 - \frac{\pi(n-3)}{\pi(n-2)} X_{n-2}, \frac{\pi(n-1)}{\pi(n-2)})} dX_{n-1} dX_{n-2} \ldots dX_{i+1}
\end{equation*}
The left-hand side of equation \ref{BdEqMarkovChain} can then be written as $\frac{f_{i+1,n}(\min(t, 1 - \frac{\pi(i)}{\pi(i+1)} a_{i}))}{f_{i+1,n}(1 - \frac{\pi(i+1)}{\pi(i)})}$. Next, define
\begin{equation*}
g_{i+1,n}(x) = \int_{X_{1} = 0}^{\min(1, \frac{\pi(2)}{\pi(1)})} \ldots   \int_{X_{i+1} = 0}^{\min(x,\frac{\pi(i+1)}{\pi(i)})} \int_{X_{i+2}=0}^{\min(1- \frac{\pi(i)}{\pi(i+1)}X_{i+1}, \frac{\pi(i+2)}{\pi(i+1)} )} \ldots \int_{X_{n-1}=0}^{\min(1 - \frac{\pi(n-2)}{\pi(n-1)} X_{n-2},\frac{\pi(n)}{\pi(n-1)})} dX
\end{equation*}
The right-hand side of equation \ref{BdEqMarkovChain} can then be written as $\frac{g_{i+1,n}(\min(t, 1 - \frac{\pi(i)}{\pi(i+1)} a_{i}))}{g_{i+1,n}(1 - \frac{\pi(i+1)}{\pi(i)} a_{i})}$. Next note that for any fixed $0 \leq x \leq 1$, $g_{i+1,n}(x) = f_{1,i-1}(1) f_{i+1,n}(x)$. Thus,
\begin{align*}
P[K[i+1,i+2] \leq t \vert  K[i,i+1] = a_{i}] &= \frac{g_{i+1,n}(\min(t, 1 - \frac{\pi(i)}{\pi(i+1)} a_{i}))}{g_{i+1,n}(1 - \frac{\pi(i+1)}{\pi(i)} a_{i})} \\
&= \frac{f_{1,i-1}(1) f_{i+1,n}(\min(t, 1 - \frac{\pi(i)}{\pi(i+1)} a_{i}))}{f_{1,i-1}(1)f_{i+1,n}(1 - \frac{\pi(i+1)}{\pi(i)})} \\
&= P[K[i+1,i+2] \leq t \vert K[1,2] = a_{1}, \ldots, K[i,i+1] = a_{i}]
\end{align*}
This proves the lemma. $\square$ \par 

As an aside, several closely related lemmas are possible. For example, fix a distribution $\pi$ on $[n]$, and a rooted tree $T = ([n], E)$. One could sample from all kernels $K$ which satisfy $K[x,y] = 0$ if $(x,y) \notin E$, and which give rise to Markov chains with stationary distribution $\pi$. Then the same factorization argument says that for any vertex $v \in T$ with unique parent $p$, child $c_{1}$, and with edges $(p_{1,1}, p_{1,2}), \ldots, (p_{m,1}, p_{m,2})$ such that any path from $p_{i,1}$ to $v$ passes through $p_{i,2}$ and then $p$,
\begin{equation*}
P[K[v, c_{1}] \leq x \left| K[p,v] = a_{0}, K[p_{i,1}, p_{i,2}  \right.] = a_{i}] = P[K[v, c_{1}] \leq x \vert K[p,v] = a_{0}]
\end{equation*}
Since less is known about general chains on trees, it is harder to apply this result to the study of cutoff. \par 

In order to use Lemma \ref{BdLemmaConcLipMin}, the following two bounds on the entries of random elements of $\mathcal{B}_{\pi}$ are necessary. First, we bound the distribution from above in terms of uniform random variables on $[0,1]$:

\begin{lemma} [Largeness] \label{LemmaBdLargenessGibbs}

Let $B$ be chosen uniformly from $\mathcal{B}_{\pi}$. For any interval $(a,b) \subset [0,1]$, 
\begin{equation*}
P \left[ B[i] \in \left(a, \frac{a+b}{2}\right) \right] \geq P\left[B[i] \in \left( \frac{a+b}{2}, b \right)\right]
\end{equation*}
This remains true when conditioned on the particular values of any collection of other entries, in the following sense. Let $\mathcal{F}_{i,j,u,v}$ be the event that $B[i] = u$ and $B[j] = v$. Then:
\begin{equation*}
P \left[B[i] \in \left(a, \frac{a+b}{2}\right) \Bigg| \mathcal{F}_{i - k_{1}, i+ k_{2}, u, v}\right] \geq P\left[B[i] \in \left( \frac{a+b}{2}, b \right) \Bigg| \mathcal{F}_{i - k_{1}, i+ k_{2}, u, v} \right]
\end{equation*}
\end{lemma} 
 To prove this, for $X \in \mathcal{B}_{\pi}$ with $X[i] = \frac{a+b}{2} + \epsilon$ and $0 < \epsilon < \frac{a+b}{2}$, define $\widehat{X}$ by $\widehat{X}[i] = \frac{a+b}{2} - \epsilon$ and $\widehat{X}[j] = X[j]$ for $j \neq i$. Then $\widehat{X} \in \mathcal{B}_{\pi}$. The map $X \rightarrow \widehat{X}$ is clearly measure-preserving, and this proves the inequality with no conditioning. The proof when conditioned on the particular values of other entries is identical - the `hat' map $X \rightarrow \widehat{X}$ is a measure-preserving map from $\{ B \in \mathcal{B}_{\pi} \, : \, B[i] \in (a, \frac{a+b}{2}), B[i-k_{1}] = v_{1},B[i + k_{2}] = v_{2} \}$ to $\{ B \in \mathcal{B}_{\pi} \, : \, B[i] \in (\frac{a+b}{2},b), B[i-k_{1}] = v_{1},B[i + k_{2}] = v_{2} \}$ . $\square$ \par 

\begin{cor} For $X$ chosen uniformly in $\mathcal{B}_{\pi}$, $P[X[i] \geq x \min(1, \frac{\pi(i+1)}{\pi(i)}) ] \leq (1-x)$ for all $x \in [0,1]$. Again, this remains true when conditioned on the particular values of any collection of other entries.
\end{cor}

By Theorem \ref{BdThmMarkovProp}, the conditioning result in Lemma \ref{LemmaBdLargenessGibbs} is quite general. \par 

Next, we bound the distribution from below in terms of uniform random variables on $[0,1]$:

\begin{lemma}[Smallness] \label{LemmaBdSmallness}
Let $X$ be chosen uniformly from $\mathcal{B}_{\pi}$, and let $\mathcal{F}_{i,j,u,v}$ be the event that $X[i] = u$ and $X[j] = v$. Then for any $D>3$, and fixed $a,b \in [0,1]$, 
\begin{equation*}
P \left[ X[i+2] < \frac{1}{D} \, \Bigg| \, \mathcal{F}_{i,i+4,a,b}  \right] \leq \frac{16}{D} \min \left( 1, \frac{\pi(i+3)}{\pi(i+2)} \right)
\end{equation*}
Furthermore, if $\pi(i+2) \leq \pi(i+3)$, then
\begin{equation*}
P \left[ X[i+2] < \frac{1}{D} \, \Bigg| \, \mathcal{F}_{i,i+4,a,b} \right] \leq \frac{3}{D-1}
\end{equation*}
\end{lemma}

Let $\alpha = \frac{1}{D}$, fix some $r>0$ to be determined later, and choose $x_{1}$ and $x_{2}$ satisfying $0 \leq x_{1} \leq \alpha$ and $r \alpha \leq x_{2} \leq (r+1) \alpha$. Then define
\begin{equation*}
L(x) = \min \left( 1 - \frac{\pi(i)}{\pi(i+1)} a, \frac{\pi(i+2)}{\pi(i+1)}(1 - x) \right)
\end{equation*}
and $R_{1} = \frac{L(x_{2})}{L(x_{1})} \leq 1$. There are three cases to consider to bound $R_{1}$. \\
\textbf{Case 1:}
\begin{equation*}
1 - \frac{\pi(i)}{\pi(i+1)} a \leq \frac{\pi(i+2)}{\pi(i+1)} (1 - x_{2}) \leq \frac{\pi(i+2)}{\pi(i+1)} (1 - x_{1})
\end{equation*}
In this case, $R_{1} = 1$. \\
\textbf{Case 2:}
\begin{equation*}
 \frac{\pi(i+2)}{\pi(i+1)} (1 - x_{2}) \leq \frac{\pi(i+2)}{\pi(i+1)} (1 - x_{1}) \leq 1 - \frac{\pi(i)}{\pi(i+1)} a 
\end{equation*}
In this case,
\begin{align*}
R_{1} &= \frac{1- x_{2}}{1 - x_{1}} \\
& \geq 1 - (r+1) \alpha 
\end{align*}
\textbf{Case 3:}
\begin{equation*}
 \frac{\pi(i+2)}{\pi(i+1)} (1 - x_{2})  \leq 1 - \frac{\pi(i)}{\pi(i+1)} a \leq \frac{\pi(i+2)}{\pi(i+1)} (1 - x_{1})
\end{equation*}
In which case
\begin{align*}
R_{1} &= \frac{\frac{\pi(i+2)}{\pi(i+1)} (1 - x_{2})}{1 - \frac{\pi(i)}{\pi(i+1)} a} \\
&\geq \frac{1 - x_{2}}{1 - x_{1}} \\
& \geq 1 - (r+1)\alpha 
\end{align*}
Thus, in all three cases, $R_{1} \geq 1 - (r+1)\alpha$. Next, define
\begin{equation*}
U(x) = \min(1 - \frac{\pi(i+2)}{\pi(i+3)}b, \frac{\pi(i+4)}{\pi(i+3)} (1 - x))
\end{equation*}
and define $R_{2} = \frac{U(x_{2})}{U(x_{1})} \leq 1$. Again, there are three cases. \\
\textbf{Case 1:}
\begin{equation*}
\frac{\pi(i+4)}{\pi(i+3)} (1 - b) \leq 1 - \frac{\pi(i+2)}{\pi(i+3)} x_{2} \leq 1 - \frac{\pi(i+2)}{\pi(i+3)} x_{1}
\end{equation*}
In this case, $R_{2} = 1$. \\
\textbf{Case 2:}
\begin{equation*}
 1 - \frac{\pi(i+2)}{\pi(i+3)} x_{2} \leq 1 - \frac{\pi(i+2)}{\pi(i+3)} x_{1} \leq \frac{\pi(i+4)}{\pi(i+3)} (1 - b) 
\end{equation*}
In this case,
\begin{align*}
R_{2} &= \frac{1 - \frac{\pi(i+2)}{\pi(i+3)} x_{2}}{1 - \frac{\pi(i+2)}{\pi(i+3)} x_{1}} \\
&\geq 1 - \frac{\pi(i+2)}{\pi(i+3)} (r+1) \alpha 
\end{align*}
\textbf{Case 3:}
\begin{equation*}
 1 - \frac{\pi(i+2)}{\pi(i+3)} x_{2} \leq \frac{\pi(i+4)}{\pi(i+3)} (1 - b)  \leq 1 - \frac{\pi(i+2)}{\pi(i+3)} x_{1} 
\end{equation*}
In this case,
\begin{align*}
R_{2} &= \frac{1 - \frac{\pi(i+2)}{\pi(i+3)} x_{2}}{\frac{\pi(i+4)}{\pi(i+3)} (1 - b)} \\
&\geq \frac{1 - \frac{\pi(i+2)}{\pi(i+3)} x_{2}}{1 - \frac{\pi(i+2)}{\pi(i+3)} x_{1}} \\
&\geq 1 - \frac{\pi(i+2)}{\pi(i+3)} (r+1) \alpha 
\end{align*}
So again, in all three cases, $R_{2} \geq 1 - \frac{\pi(i+2)}{\pi(i+3)} (r+1) \alpha $. \par 
Next, note that $R_{1}$ gives the ratio of the length of the interval of allowed values of $X[i+1]$ given $X[i] = a$ and $X[i+2] = x_{2}$ to the length of the interval of allowed values values of $X[i+1]$ given $X[i] = a$ and $X[i+2] = x_{1}$. Similarly, $R_{2}$ gives the ratio of the length of the interval of allowed values values of $X[i+3]$ given $X[i+4] = b$ and $X[i+2] = x_{2}$ to the length of the interval of allowed values values of $X[i+3]$ given $X[i+4] = b$ and $X[i+2] = x_{1}$. The product of these two ratios is at least $(1 - (r+1) \alpha)( 1 - \frac{\pi(i+2)}{\pi(i+3)} (r+1) \alpha )$. Thus, it is possible to construct a map $\phi = (\phi_{1}, \phi_{2}, \phi_{3})$ from the set of triples $(X[i+1], X[i+2], X[i+3])$ with $X[i+2] = x_{2}$ to the set of triples with $X[i+2] = x_{1}$ which is linear in the first and third coordinates, and where the product of the slopes of the two linear parts are at most $(1 - (r+1) \alpha)( 1 - \frac{\pi(i+2)}{\pi(i+3)} (r+1) \alpha )$. This implies that 

\begin{equation} \label{IneqSmallnessUtility}
\frac{P[X[i+2] \in [r \alpha, (r+1)\alpha] \big| \mathcal{F}_{i,i+4,a,b}]}{ P[X[i+2] \leq \alpha \big| \mathcal{F}_{i,i+4,a,b}]} \geq (1 - (r+1) \alpha)( 1 - \frac{\pi(i+2)}{\pi(i+3)} (r+1) \alpha )
\end{equation}

Assume first that $\pi(i+3) < \pi(i+2)$. Then:
\begin{equation*}
P[X[i+2] \geq \alpha \big| \mathcal{F}_{i,i+4,a,b}]  \geq \sum_{r=1}^{\frac{\pi(i+3)}{\pi(i+2) \alpha} - 1} P[X[i+2] \in [r \alpha, (r+1) \alpha] \big| \mathcal{F}_{i,i+4,a,b}] \\
\end{equation*}
Using inequality \eqref{IneqSmallnessUtility} relating $P[X[i+2] \in [r \alpha, (r+1) \alpha] \big| \mathcal{F}_{i,i+4,a,b}]$ to $P[X[i+2] \leq \alpha \big| \mathcal{F}_{i,i+4,a,b}]$,  summing only the first half of the terms, and bounding each of those terms by their minimum results in the bound
\begin{equation*}
P[X[i+2] \geq \alpha \big| \mathcal{F}_{i,i+4,a,b}]  \geq \frac{4}{\lfloor \frac{1}{2 \alpha} \frac{\pi(i+3)}{\pi(i+2)} - 1 \rfloor} P[X[i+2] \leq \alpha \big| \mathcal{F}_{i,i+4,a,b}]
\end{equation*}

This proves the first inequality for $\pi(i+2) > \pi(i+3)$. \par 
If, however, $\pi(i+3) > \pi(i+2)$, $P[X[i] \in [r \alpha, (r+1)\alpha] \big| \mathcal{F}_{i,i+4,a,b}] \geq (1 - ra)^{2} P[X[i+2] \leq \alpha\big| \mathcal{F}_{i,i+4,a,b}]$, and the above calculation can be made a little more careful: 
\begin{align*}
P[X[i+2] \geq \alpha\big| \mathcal{F}_{i,i+4,a,b}] &\geq \sum_{r=1}^{\alpha^{-1} - 1} P[X[i+2] \in [r\alpha, (r+1)\alpha] \big| \mathcal{F}_{i,i+4,a,b}] \\
&\geq \sum_{r=1}^{\alpha^{-1}-1} (1 - r\alpha)^{2} P[X[i+2] \leq \alpha \big| \mathcal{F}_{i,i+4,a,b}] \\
&\geq P[X[i+2] \leq \alpha \big| \mathcal{F}_{i,i+4,a,b}] \sum_{r=1}^{\alpha^{-1} - 1} (r\alpha)^{2} \\
&\geq P[X[i+2] \leq \alpha \big| \mathcal{F}_{i,i+4,a,b}] \left( \frac{1}{3\alpha} - \frac{1}{3} \right)
\end{align*}
So in fact $P[X[i] \leq \alpha] \leq 3 \frac{\alpha}{1-\alpha}$. $\square$ \\

We note that, for $a \in [\frac{1}{2}, 1]$, very similar calculations show:

\begin{equation} \label{EqAlmostSmallness}
P \left[ X[i+1] < \frac{1}{D} \, \big| \, \mathcal{F}_{i,i+3,a,b} \right] = O\left( \frac{1}{D} \min \left( 1, \frac{\pi(i+3)}{\pi(i+2)} \right) \right)
\end{equation}

We are finally ready to prove Lemma \ref{BdLmGeneralSdMixing}. By lemmas \ref{LemmaBdLargenessGibbs} and \ref{LemmaBdSmallness}, for $B$ chosen uniformly from $\mathcal{B}_{\pi}$, the distribution $P[X[i+2] \leq x \vert X[i] = a, X[i+4] = b]$ satisfies the conditions of Lemma \ref{BdLemmaConcLipMin} with constant $C = C(i)$. 
Next, fix $1 \leq i \leq j \leq n$ with $i + j \leq n + 2$. Let $Z$ and $Q$ be chosen from $\mathcal{B}_{\pi}$ chosen uniformly conditioned on $Q[j] = b_{1}$, $Z[j] = b_{2}$ and $Z[i] = Q[i] = a$. We will apply the results of Lemma \ref{BdLemmaConcLipMin} to successive blocks of size two. That is, by Theorem \ref{BdEqMarkovChain}, we can view $\{ Z[\ell] \}_{\ell = i}^{j}$ and $\{ Q[\ell] \}_{\ell = i}^{j}$ as Markov chains, and so $\{ Z[2 \ell] \}_{i \leq 2\ell \leq j}$ and $\{ Q[\ell] \}_{i \leq 2\ell \leq j}$ are also Markov chains. We will try to force these two chains to coallesce. Since they are Markov chains, if $\tau$ is the (random) coallescence time, we have by Lemma \ref{LemmaFundCoup}
\begin{equation*}
\vert \vert \mathcal{L}( Z[i+2 \ell]) - \mathcal{L}( Q[i+2 \ell]) \vert \vert_{TV} \leq P[\tau > \ell]
\end{equation*}

The minorization described in Lemma \ref{BdLemmaConcLipMin} implicitly describes a one-step coupling of the two chains, which is the coupling we use. The lemma bounds the probability that they coallesce in each step. Since $\tau > \ell$ only if coallescence fails at each of the first $\ell$ steps, this proves the lemma.  $\square$. \par 

As an aside, Theorem $6.1.v$ of \cite{DiWo10} gives a much better mixing estimate in the case that $\pi$ is uniform. Their estimate is based on finding all of the eigenvalues of the limiting Markov kernel as $n$ goes to $\infty$, which is not practical in the general case. Lemma 20 of \cite{Smit12a} gives intermediate bounds in the case that $\pi$ is monotone but not uniform.  \par 

We will use the sub-diagonal representation of tri-diagonal matrices which is analogous to the representation in equation \eqref{EqSupDiagRep} to find a different improvement. Note that $R(\ell) \geq \frac{4}{27}$ if $\pi(\ell + 3) \geq \pi(\ell + 2)$. An analogous bound for mixing of the sub-diagonal entries holds, with $R(\ell) \geq \frac{4}{27}$ if $\pi(\ell + 3) \leq \pi(\ell + 2)$. This leads to the following corollary:

\begin{cor} \label{BdCorGenMixing}

Let $Z$ and $Q$ be as in Lemma \ref{BdLmGeneralSdMixing}. Then
\begin{equation*}
\vert \vert \mathcal{L}(Z[i+4 \ell]) - \mathcal{L}(Q[i+4 \ell]) \vert \vert_{TV} \leq \left( 1 - \frac{4}{27} \right)^{\ell}
\end{equation*}
\end{cor}
\par 

We briefly sketch the argument. We will look at the Markov chains $\{ Z[i + 4 \ell] \}_{0 \leq 4 \ell \leq j - i}$ and $\{ Q[i + 4 \ell] \}_{0 \leq 4 \ell \leq j - i}$. By symmetry the bound in Lemma \ref{BdLmGeneralSdMixing} also applies to the parameterization of $\mathcal{A}_{\pi}$ by subdiagonal entries, with a term of the form $16 \min(1, \frac{\pi(\ell)}{\pi(\ell + 1)})$ in place of the term $16 \min(1, \frac{\pi(\ell+1)}{\pi(\ell)})$ from Lemma \ref{BdLmGeneralSdMixing}. That is, the ratio of successive terms of $\pi$ is flipped. Thus, in each block of size 4, we can attempt a coallescing step using the one-step coupling described in Lemma \ref{BdLemmaConcLipMin} or the analogous version for sub-diagonal entries. At least one of these has a success probability of at least $\frac{4}{27}$. The blocks are of length 4 rather than 2 to allow space to switch between the coupling in the super- and sub-diagonal representations, as the 1-step couplings are different. $\square$ \par 

With this corollary, we can now prove our contraction estimate for starting distributions $X_{0}, Y_{0}$ differing at the single entry $j$ satisfying $k < j < n - k$. Assume the block $(j+1, \ldots, j+k)$ is being updated at time 0. The method is to choose entries inside the blocks in groups of size four, sequentially conditioned upon the endpoints of the large block. They are coupled as described in the proof of Lemma \ref{BdLmGeneralSdMixing}. As shown in Corollary \ref{BdCorGenMixing}, at each such step, $X_{1}$ and $Y_{1}$ couple with probability at least $\frac{4}{27}$. Thus, the expected increase in distance is at most $\sum_{i} 4 (\frac{23}{27})^{i} = 27$, uniformly in $k$. Under this coupling, then, 
\begin{equation} \label{IneqMiddleContract}
 E[d(X_{1}, Y_{1}) \vert d(X_{0}, Y_{0}) = 1, X_{0}[j] \neq Y_{0}[j]] \leq 1 - \frac{1}{n + 2w - 2}(k - 54)
 \end{equation}

for $k<j<n-k +1$. By the same argument, for $j \leq k$ or $j \geq n - k +1$, we have
\begin{equation} \label{IneqEdgeContract}
 E[d(X_{1}, Y_{1}) \vert d(X_{0}, Y_{0}) = 1, X_{0}[j] \neq Y_{0}[j]] \leq 1 - \frac{1}{n + 2w - 2}(w - 54)
 \end{equation}

Choosing $k = w = 55$, by inequalities \eqref{IneqMiddleContract} and \eqref{IneqEdgeContract} we have for any $1 \leq j \leq n$

\begin{equation} \label{IneqEdgeContract}
 E[d(X_{1}, Y_{1}) \vert d(X_{0}, Y_{0}) = 1, X_{0}[j] \neq Y_{0}[j]] \leq 1 - \frac{1}{n + 108}
 \end{equation}

And so for $n \geq 108$, by Theorem \ref{BdThmPathCoup} above,
\begin{equation*}
E[d(X_{t}, Y_{t})] \leq n \left(1 - \frac{1}{2n} \right)^{t}
\end{equation*}
By Lemma \ref{LemmaFundCoup} and then Markov's inequality,
\begin{align*}
\vert \vert \mathcal{L}(X_{t}) - U \vert \vert_{TV} &\leq P[d(X_{t}, Y_{t}) \geq 1] \\
&\leq 2 n \left(1 - \frac{1}{2n} \right)^{t}
\end{align*}
proving the upper bound. \par 
The lower bound follows from the usual `coupon collector' problem. Since our walk is over a continuous space, the total variation distance to stationarity of the Markov chain at time $t$ must be at least the probability that not all coordinates have been updated by time $t$. Since only $k$ coordinates are updated at a time, the classical coupon-collector results in \cite{ErRe61} tell us that at time $T = \frac{1}{k} n (\log n - c)$, $\sup_{A \in \Sigma} \vert K_{n}^{T}(x, A) - \pi(A) \vert \geq 1 - \exp(-\exp(c)) + o(1)$ as $n$ goes to infinity. $\square$ \par

\section{Cutoff Preliminaries} \label{SecCutoffBasics}

We now return to the main problem, determining when sequences of random chains exhibit cutoff. The proof will rely on results in Section \ref{SecGibbs1} and their analogues in Section 6 of \cite{DiWo10}. The approach makes heavy use of coupling in order to apply the results of \cite{DiWo10} to much less symmetric objects. Both \cite{DiWo10} and this paper heavily rely on Theorem \ref{BdThmCutoffCrit} above, which states that there is cutoff for a sequence of birth and death chains if and only if the product of the mixing time and spectral gap goes to infinity. So, the problem reduces to finding good bounds on the mixing time and spectral gap of these chains. \par 

We start by recalling a few results giving good estimates of the mixing time and spectral gap for birth and death chains. Let $m$ be the median of $\pi$; for $\pi$ symmetric this will be $\frac{n}{2}$. From \cite{Micl99}, the spectral gap $(1 - \lambda_{n}(K))$ for a BD chain with transition kernel $K$ satisfies
\begin{equation*}
\frac{1}{4B} \leq (1 - \lambda_{n}(K)) \leq \frac{2}{B}
\end{equation*}
where $B = \max ( B_{+}(m), B_{-}(m) )$ and
\begin{align} \label{EqSpectralGapBounders}
B_{+}(m) &= \max_{x > m} \left(\sum_{y = m+1}^{x} \frac{1}{\pi(y) K(y,y-1)}\right)\sum_{y>x} \pi(y) \\
&= \max_{x > m} \left( \sum_{y = m+1}^{x} \frac{1}{\pi(y-1) K(y-1,y)}\right)\sum_{y>x} \pi(y) \\
B_{-}(m) &= \max_{x < m} \left(\sum_{y = x}^{m-1} \frac{1}{\pi(y) K(y,y+1)}\right)\sum_{y<x} \pi(y) 
\end{align}

The mixing time $\tau_{mix}(\frac{1}{4})$ will be estimated by $\max(E_{0}[T_{m}], E_{n}[T_{m}])$, where $E_{i}[T_{j}]$ is the expected hitting time of $j$ for the birth and death chain started at position $i$. Two results are needed. The first is that there exist universal functions $c_{1}(\delta)$, $c_{2}(\delta)$ so that for quantile functions $m(\delta) = \inf \{k : \sum_{j \leq k} \pi(j) \geq \delta \}$,
\begin{equation} \label{BdIneqHittingMixingRelation}
c_{1}(\delta) \max(E_{0}[T_{m(\delta)}], E_{n}[T_{m(1 - \delta)}]) \leq \tau_{mix}(\frac{1}{4}) \leq c_{2}(\delta) \max(E_{0}[T_{m(\delta)}], E_{n}[T_{m(1 - \delta)}])
\end{equation}
for any $\frac{1}{2} < \delta < 1$. This is an immediate corollary of Theorem 1.1 of \cite{PeSo11}. The inequalities \ref{BdIneqHittingMixingRelation} will be used to find the rough order of the mixing time by looking instead at the hitting times. \par 

The second theorem deals with locating cutoff when it occurs. In \cite{DLP08}, it is shown that if a sequence of chains with mixing profiles $\tau_{n}(\epsilon)$ and median functions $m_{n}(\delta)$ exhibits cutoff, then
\begin{equation} \label{BdEqCutoffLocationEqHittingTime}
\lim_{n \rightarrow \infty} \frac{\tau_{n}(\frac{1}{4})}{\max(E_{0}[T_{m_{n}(\delta)}], E_{n}[T_{m_{n}(1 - \delta)}])} = 1
\end{equation}
This will be used to estimate cutoff location in the few cases it is possible to prove the existence of cutoff. \par 

The other basic result is the following classical formula for expected hitting times (see e.g. \cite{KrSc93} \cite{PaTe96}):
\begin{align} \label{BdEqHittingTime}
E_{j-1}[T_{j}] &= \frac{1}{\pi(j-1) K[j-1,j]} \sum_{q=0}^{j-1} \pi(q) \\
E_{i}[T_{j}] &= \sum_{v=0}^{j-i-1} E_{i+v}[T_{i+v+1}] \\
&= \sum_{v=0}^{j-i-1} \frac{1}{\pi(i+v) K[i+v,i+v+1]} \sum_{q=0}^{i+v} \pi(q)
\end{align}
Thus, estimating both the mixing and relaxation time reduces to estimating the weighted harmonic sums of the super-diagonal entries of the transition matrix found in equations \eqref{EqSpectralGapBounders} and \eqref{BdEqHittingTime}. This requires some information about the marginal distribution of each entry, a theorem showing that these entries are not too dependant on each other, and finally some sort of invariance principle. The information about the marginal distributions is in Lemmas \ref{LemmaBdLargenessGibbs} and \ref{LemmaBdSmallness}, with a more precise bound given by Theorem $6.1.v$ of \cite{DiWo10} for specific chains. The information about mixing is in Lemma \ref{BdLmGeneralSdMixing}, and again Theorem $6.1.i$ of \cite{DiWo10} gives more precise information for specific chains. The invariance principle will be Theorem 3.4 of \cite{BKS10}, which requires some preparation to state.  \par 

Recall that a stochastic process $X_{n}$ is called strictly stationary if $(X_{i_{1}}, X_{i_{2}}, \ldots, X_{i_{k}}) \stackrel{D}{=} (X_{i_{1}+s}, X_{i_{2}+s}, \ldots, X_{i_{k}+s}) $ in distribution for all $s,k \geq 1$ and all $i_{1} < i_{2} < \ldots < i_{k}$. Such a sequence is called \textit{jointly regularly varying} with index $\alpha$ if, for all $k \geq 1$, there exists a random measure $\Theta_{k}$ on the sphere $S^{k-1} = \{ (v_{1}, \ldots, v_{k}) \vert \, \sum_{i=1}^{k} v_{i}^{2} = 1 \}$ so that for all $u \in (0, \infty)$ and measurable $A \subset S^{k-1}$,
\begin{equation*}
\lim_{x \rightarrow \infty} \frac{P[ \vert \vert (X_{1}, \ldots, X_{k}) \vert \vert_{2} > ux, \, \frac{(X_{1}, \ldots, X_{k})}{\vert \vert (X_{1}, \ldots, X_{k}) \vert \vert_{2}} \in A]}{P[\vert \vert X \vert \vert_{2} > x]} = u^{-\alpha} P[\Theta_{k} \in A]
\end{equation*}
This condition is hard to verify directly for the stationary sequence used in this paper. Theorem 2.1 of \cite{BaSe09} gives a useful sufficient condition:
\begin{thm}[Condition for Joint Regular Variation] \label{BdThmSuffCondJointRegVar} 
Let $X_{n}$ be a strictly stationary stochastic process. If there exists another process $Y_{n}$ such that
\begin{itemize}
\item $P[\vert Y_{0} \vert > y] = y^{-\alpha}$ for $y \geq 1$
\item The following limit holds in finite-dimensional distribution:
\begin{equation*}
\lim_{x \rightarrow \infty} \{ (x^{-1} X_{n})_{n \in \mathbb{Z}} \vert \, \, \vert X_{0} \vert > x \} = Y_{n}
\end{equation*}
\end{itemize}
then $X_{n}$ is jointly regularly varying with exponent 1.
\end{thm}

For a process $Y_{n}$ of this form, define $\theta_{n} = \frac{Y_{n}}{\vert Y_{0} \vert}$ to be the \textit{tail process} of $Y_{n}$ (and also of $X_{n}$). \par 

Fix some sequence $a_{n}$ satisfying $\lim_{n \rightarrow \infty} n P[X_{1} > a_{n}] = 1$. Say that a stochastic sequence $X_{n}$ satisfies condition 1 if there exists a positive integer sequence $r_{n}$ so that $r_{n} \rightarrow \infty$, $n^{-1} r_{n} \rightarrow 0$, and for all $u > 0$,
\begin{equation*}
\lim_{m \rightarrow \infty} \limsup_{n \rightarrow \infty} P[ \max_{m \leq \vert i \vert \leq r_{n}} \vert X_{i} \vert > u a_{n} \vert \, \vert X_{0} \vert > u a_{n}] = 0
\end{equation*}
Define $\sigma_{a}^{b}$ to be the $\sigma$-algebra generated by $\{ X_{i} \}_{a \leq i \leq b}$. Say that a strictly stationary sequence $X_{n}$ satisfies condition 2 if 
\begin{equation*}
\lim_{m \rightarrow \infty} \sup_{E \in \sigma_{- \infty}^{0}, F \in \sigma_{m}^{\infty}} \vert P[E \cap F] - P[E] P[F] \vert = 0
\end{equation*}
This condition is known in the literature as \textit{strong mixing}. Say that $X_{n}$ satisfies condition 3 if, for all $\delta > 0$,
\begin{equation*}
\lim_{u \rightarrow 0} \limsup_{n \rightarrow \infty} P \left[ \max_{0 \leq k \leq n} \big| \sum_{i=1}^{k} \left( \frac{X_{i}}{a_{n}} \textbf{1}_{\vert X_{i} \big| \leq u a_{n}} - E \left[ \frac{X_{i}}{a_{n}} \textbf{1}_{\vert X_{i} \vert \leq u a_{n}} \right] \right) \vert > \delta \right] = 0
\end{equation*}
By Proposition 3.7 of \cite{BKS10}, $X_{n}$ satisfies condition 3 if it is $\rho$-mixing. That is, if 
\begin{equation*}
\lim_{m \rightarrow \infty} \sup \{ \vert E[YZ] - E[Y]E[Z] \vert : Y \in L^{2}(\sigma_{-\infty}^{0}), Z \in L^{2}(\sigma_{m}^{\infty}), \vert \vert Y \vert \vert_{2} = \vert \vert Z \vert \vert_{2} = 1 \} = 0
\end{equation*}
Theorem 3.4 of \cite{BKS10} states that

\begin{thm}[Functional Limit Theorem for Mixing Sequences] \label{BdTheoremFunctClt}

Let $X_{n}$ be a strictly stationary stochastic process which is jointly regularly varying with exponent 1. Further, assume that it satisfies conditions 1, 2 and 3. Let $\Theta_{n}$ be its tail process, and assume that $\Theta_{n}$ almost surely has no two values of opposite sign. Then the process
\begin{equation*}
V_{n}(t) = \sum_{k=1}^{\lfloor nt \rfloor} \frac{X_{k}}{a_{n}} - \lfloor nt \rfloor E \left[\frac{X_{1}}{a_{n}} \textbf{1}_{\vert X_{1} \vert \leq a_{n}} \right]
\end{equation*}
Converges in the $M_{1}$ topology on $D[0,1]$ (see \cite{Skor56} for a definition of this topology) to an $\alpha$-stable Levy process with Levy triple $(0, \nu, b)$ (see pp. 150 of \cite{Resn07} for a definition of a process with given Levy triple), where $\nu$ is given by:
\begin{equation*}
\lim_{u \rightarrow 0} \nu^{(u)} = \nu
\end{equation*}
and where, for $x>0$, $\nu^{(u)}$ is given by:
\begin{align*}
\nu^{(u)}(x, \infty) & = u^{-1} P \left[u \sum_{i \geq 0} Y_{i} \textbf{1}_{\vert Y_{i} \vert > 1} > x, \sup_{i \leq -1} \vert Y_{i} \vert \leq 1 \right] \\
\nu^{(u)}(-\infty, -x) & = u^{-1} P \left[u \sum_{i \geq 0} Y_{i} \textbf{1}_{\vert Y_{i} \vert > 1} < -x, \sup_{i \leq -1} \vert Y_{i} \vert \leq 1 \right] 
\end{align*}
and $b$ is given by
\begin{equation*}
b = \lim_{u \rightarrow 0} \int_{x: u < \vert x \vert \leq 1} x \nu^{(u)} (dx) - \int_{x: u < \vert x \vert \leq 1} x \mu(dx)
\end{equation*}
where $\mu$ is given by
\begin{equation*}
\mu(dx) = (p \textbf{1}_{(0, \infty)} (x) + q \textbf{1}_{(-\infty, 0}(x)) \vert x \vert^{-2} dx
\end{equation*}
and, finally,
\begin{align*}
p &= \lim_{x \rightarrow \infty} \frac{P[X_{1} > x]}{P[\vert X_{1} \vert > x]} \\
q &= \lim_{x \rightarrow \infty} \frac{P[X_{1} < -x]}{P[\vert X_{1} \vert > x]} 
\end{align*}
\end{thm}

Almost all of the conditions for this theorem are close to holding for the stochastic process $X_{i} = B[i]$ for $B$ chosen uniformly from $\mathcal{B}_{\pi}$. The one exception is stationarity; $X_{i}$ must be a time-homogeneous Markov chain for the theorem to apply. If $\pi(i) = a^{i-1} \frac{1-a}{1 - a^{n}}$, $X_{i}$ will `look like' a stationary time-homogeneous chain for entries far from $1$ and $n$ (see the limiting process defined in \cite{DiWo10} for a precise example). Otherwise, the Markov chain will not be time-homogeneous, nevermind close to stationarity. To get around this problem, we will couple the process of interest, $X_{i} = B[i]$, to the stationary limiting process $Z_{i}$ described in Section 6 of \cite{DiWo10}. 

To review that paper, the process $Z_{i}$ is the weak limit of the stochastic processes $Z_{i}^{(n)} = A_{n}[i]$, where $A_{n}$ is drawn uniformly from $\mathcal{B}_{(\frac{1}{n}, \frac{1}{n}, \ldots, \frac{1}{n})}$. By Theorem 6.5 of \cite{DiWo10}, this is time-homogeneous with stationary distribution given by the cumulative distribution function $P[Z_{j} \leq x] = \sin(\frac{\pi x}{2})$, and transition kernel $P[Z_{j+1} \leq z \vert Z_{j} = x] = \frac{\sin(\frac{\pi}{2}\min(z, 1-x))}{\sin(\frac{\pi}{2}(1-x) )}$. Theorem \ref{BdTheoremFunctClt} will generally be applied to the process $Z_{j}$, and a coupling will be found to compare $X_{j}$ and $Z_{j}$. The exception is that Theorem \ref{BdTheoremFunctClt} may be applied directly to $X_{i}$ when $\pi$ is of the form $\pi(i) = a^{i-1} \frac{1-a}{1 - a^{n}}$. Although Theorem \ref{BdTheoremFunctClt} will be applied to $Z_{j}$ more often than $X_{j}$ in this paper, many of the bounds depend only on Lemmas \ref{LemmaBdLargenessGibbs} and \ref{LemmaBdSmallness} and Corollary \ref{BdCorGenMixing}. Since $Z_{j}$ is a limit of chains drawn from $\mathcal{B}_{(\frac{1}{n}, \ldots, \frac{1}{n})}$, it is easy to check by taking limits that all of these lemmas apply to $Z_{j}$ as well. \par 
In the remainder of this section, we will show that Theorem \ref{BdTheoremFunctClt} applies to the sequence $Z_{j}$, and then prove a comparison between $X_{j}$ and $Z_{j}$. First, the application to $Z_{j}$: 

\begin{thm} [Limiting Process for Harmonic Sum of Super-Diagonal Entries] \label{BdThmLimitingSumsSupDiag}
Let $Z_{j}$ be the limiting process described in Section 6 of \cite{DiWo10}. Then the renormalized process
\begin{equation*}
V_{n}(t) = \sum_{k=1}^{\lfloor nt \rfloor} \frac{Z_{k}}{2n} - \lfloor nt \rfloor \int_{\frac{1}{2n}}^{1} \frac{1}{ny} \cos^{2} \left( \frac{\pi}{2} y \right) dy
\end{equation*}
converges in the $M_{1}$ topology to a Levy process with Levy triple $(0, \nu, 0)$ where, for $x > 1$,
\begin{equation*}
\nu(x, \infty) = \frac{1}{x}
\end{equation*}
\end{thm}

First, some initial calculations. By Theorem 6.5 of \cite{DiWo10}, the sequence $a_{n} = 2n$ satisfies
\begin{align*}
\lim_{n \rightarrow \infty} n P[Z_{1} > a_{n}] &= \lim_{n \rightarrow \infty} n \int_{0}^{\frac{1}{2n}} 2 \cos^{2} \left( \frac{\pi}{2} x \right) dx \\
&= 1
\end{align*}
So this sequence $a_{n}$ satisfies the requirements of Theorem \ref{BdTheoremFunctClt}. The term $\int_{\frac{1}{2n}}^{1} \frac{1}{ny} \cos^{2}(\frac{\pi}{2} y) dy$ in the statement of Theorem \ref{BdThmLimitingSumsSupDiag} comes from explicitly calculating $E[\frac{Z_{1}}{2n} 1_{\vert Z_{1} \vert < 2n}]$, based again on Theorem 6.5 of \cite{DiWo10}. Note that asymptotically, this is approximated by
\begin{equation*}
E \left[\frac{Z_{1}}{2n} 1_{\vert Z_{1} \vert < 2n} \right] = \frac{\log(2n)}{n} + O \left(\frac{1}{n} \right)
\end{equation*}
Next, it is necessary to show that $Z_{j}$ follows conditions 1,2 and 3. For condition 2, set $r_{n} = \sqrt{n}$. Then by lemmas \ref{LemmaBdLargenessGibbs} and \ref{LemmaBdSmallness},
\begin{align*}
\lim_{m \rightarrow \infty}  \limsup_{n \rightarrow \infty} P[\sup_{m \leq i \leq r_{n}} \vert Z_{i} \vert > u n \mid Z_{0} > u n]  & \leq \lim_{m \rightarrow \infty} \limsup_{n \rightarrow \infty} (\sqrt{n} - m) \frac{3}{un} \\
& \leq \lim_{m \rightarrow \infty} \frac{3}{4 u m} \\
&= 0
\end{align*}
so condition 1 is satisfied. Condition 2 is an immediate consequence of Lemma \ref{BdCorGenMixing}. Condition 3 is, as mentioned, implied by $\rho$-mixing, which is an immediate consequence of Theorem 6.5 of \cite{DiWo10}. Next, we describe the regular variation of $Z_{j}$. First, we have:

\begin{lemma} [Regular Variation] \label{BdLemmaRegVar}
Fix $\pi$, and let $K$ be chosen uniformly from $\mathcal{A}_{\pi}$. Under the uniform distribution, the random variable $\frac{1}{K[i,i+1]}$ is regularly varying with exponent 1 for all $0 \leq i \leq n-1$.
\end{lemma}
\begin{proof}
By Lemma \ref{LemmaBdLargenessGibbs}, $f(x) \equiv x P[K(i,i+1) < \frac{1}{x}]$ is a monotone increasing function in $x$. By Lemma \ref{LemmaBdSmallness}, $f(x)$ is bounded above by some constant $C$. Thus, $\lim_{x \rightarrow \infty} f(x)$ exists; call this value $\beta$. The next step is to show that $ \lim_{x \rightarrow \infty} \frac{ax P[\frac{1}{K(i,i+1)} > ax]}{x P[\frac{1}{K(i,i+1)} > x]} = 1$ for each $a > 0$. \\
Fix $a > 0$ and level of approximation $\delta > 0$. Then choose $X_{a, \delta}$ so that for $x > X_{a, \delta}$, $\vert x P[K(i,i+1) < \frac{1}{x}] - \beta \vert$ and $\vert ax P[K(i,i+1) < \frac{1}{ax}] - \beta \vert$ are both less than $\min(\frac{\delta \beta}{3}, \frac{1}{4})$. Then
\begin{align*}
\frac{\beta - \epsilon}{\beta + \epsilon} &\leq \frac{ax P[\frac{1}{K(i,i+1)} > ax]}{x P[\frac{1}{K(i,i+1)} > x]} \leq \frac{\beta + \epsilon}{\beta - \epsilon} \\
1 - \frac{3 \epsilon}{\beta} &\leq \frac{ax P[\frac{1}{K(i,i+1)} > ax]}{x P[\frac{1}{K(i,i+1)} > x]} \leq 1 + \frac{3 \epsilon}{\beta} \\
1 - \delta &\leq \frac{ax P[\frac{1}{K(i,i+1)} > ax]}{x P[\frac{1}{K(i,i+1)} > x]} \leq 1 + \delta 
\end{align*}
which proves the lemma. \end{proof}

We use this to show joint regular variation:

\begin{lemma}[Joint Regular Variation of Limiting Super-Diagonal Chain] \label{BdLemmaJointRegVarUnif}
$Z_{j}$ is jointly regularly varying, with a.s. nonnegative tail process.
\end{lemma}
Define $Q_{0}$ to have the distribution $P[Q_{0} > x] = \frac{1}{x}$ for $x \geq 1$, and define $Q_{j} \equiv 0$ for $j>0$. It is necessary to show only that this process satisfies the conditions of Lemma \ref{BdThmSuffCondJointRegVar}. By Lemmas \ref{LemmaBdLargenessGibbs} and \ref{LemmaBdSmallness}, for $j \geq 1$, $P[Z_{j} > a \vert Z_{0} > x] = O(a^{-1})$ as $x$ goes to infinity.  By the union bound, then, for any $a > 0$ and any fixed collection of positive integers $J$ not containing 0, $P[\sup_{j \in J} Z_{j} > xa \vert Z_{0} > x] = O(\frac{\vert J \vert}{x})$. For $j = 0$, Lemma \ref{BdLemmaRegVar} implies that $\lim_{x \rightarrow \infty} P[Z_{0} > ax \vert Z_{0} > x] = a^{-1}$ for $a \geq 1$. \par 
Note also that the process is almost surely nonnegative, and that $\theta_{n} = \textbf{1}_{n=0}$ is the (nonrandom) tail process associated with $Z_{j}$. $\square$ \par 

So all of the conditions of Theorem \ref{BdTheoremFunctClt} are satisfied for the chain $Z_{j}$. The only remaining work is to calculate the values of the Levy triple.  \par 
Define $Q_{0}$ to have the distribution $P[Q_{0} > x] = \frac{1}{x}$ for $x \geq 1$, and define $Q_{j} \equiv 0$ for $j>0$. Then 
\begin{align*}
\nu(x, \infty) &= \lim_{u \rightarrow 0} u^{-1} P[u \sum_{i \geq 0} Q_{i} 1_{Q_{i} > 1} >x, \sup_{i \leq -1} Q_{i} \leq 1] \\
&= \lim_{u \rightarrow 0} u^{-1} P[Q_{0} > u^{-1} x] \\
&= \frac{1}{x} 
\end{align*}
for $x \geq 1$. This implies $\nu(-\infty,x) = 0$ for $x < 0$. It is immediate that $b = 0$. This proves the theorem. $\square$

Finally, we prove a comparison result. For a fixed distribution $\pi$, define 
$M(i) = 16 \min(1, \frac{\pi(i+1)}{\pi(i)})$. Then:

\begin{lemma} [Coupling to Stationarity] \label{BdLemCoupToStat}
Let $B$ be drawn uniformly from $\mathcal{B}_{\pi}$, and let $X_{j} = B[j]$. Then let $A$ be drawn uniformly from $\mathcal{B}_{(\frac{1}{n}, \frac{1}{n}, \ldots, \frac{1}{n})}$, and let $Z_{j} = A[j]$. It is possible to construct four chains $Y_{j}^{(i)}$, $1 \leq i \leq 4$, such that $Y_{j}^{(i)} \stackrel{D}{=} Z_{j}$ in distribution, and so that for indices $2 \leq j \leq n-2$,

\begin{align*}
\frac{1}{256} M(2j) Y_{2j}^{(1)} &\leq X_{2j} \leq 256 M(2j) Y_{2j}^{(2)} \\
\frac{1}{256} M(2j+1) Y_{2j+1}^{(3)} &\leq X_{2j+1} \leq 256 M(2j+1) Y_{2j+1}^{(4)} \\
\end{align*}
otherwise. Note that the four chains $Y_{j}^{i}$ may not be independent.
\end{lemma}
\textbf{Proof:}
By Lemmas \ref{LemmaBdLargenessGibbs} and \ref{LemmaBdSmallness}, 

\begin{equation} \label{BdIneqCoupStat}
\frac{M(j) \beta}{256 \alpha} \leq \frac{P[\alpha X_{j} \leq x \vert X_{j-2}, X_{j+2}]}{P[\beta Y_{j} \leq x \vert Y_{j-2}, Y_{j+2}]} \leq \frac{256 M(j) \beta}{\alpha}
\end{equation}
To build $Y_{j}^{(2)}$ step by step, begin by coupling $Y_{0}^{(2)}$ to $X_{0}$ so that they are both at the same quantile in their respective distributions. That is, choose $X_{0} = x$ from its distribution, then set $Y_{0}^{(2)}$ to the unique value $a$ which satisfies $P[Y_{0}^{(2)} \leq a] = P[X_{0} \leq x]$. Next, couple $Y_{2j+2}^{(2)}$ to $X_{2j+2}$ given $X_{2j}$ and $Y_{2j}^{(2)}$ so that they both have the same quantile in their respective conditional distributions. By the inequalities on line \ref{BdIneqCoupStat}, they will satisfy the inequalities in the statement of this lemma. Do this until all even entries have been filled in. Then construct the odd entries of $Y_{j}^{(2)}$ conditionally on the even entries, independently of $X_{j}$; their values are not relevant to this lemma. The remaining sequences $Y_{t}^{(1)}, Y_{t}^{(3)}$ and $Y_{t}^{(4)}$ are constructed the same way.  $\square$
\par 
Note that if finer control of $X_{t}$ by chains with the same distribution as $Y_{t}$ is necessary, it is possible to improve the factors of $256$ in the statement of this lemma by using $2k$ chains $Y_{t}^{1}, \ldots, Y_{t}^{2k}$ and controlling entries $X_{kj + b}$ with the chains $Y_{t}^{2b-1}, Y_{t}^{2b}$. This can quantitatively improve many later bounds, and lead to improved bounds on cutoff location, but is not useful for proving existence or non-existence of cutoff. See section 2.11 of \cite{Smit12a} for details. \par 

By Lemma \ref{BdLemCoupToStat}, for any nonnegative function $h$,
\begin{align*}
\sum_{j=1}^{\frac{n}{2}} \left( \frac{h(2j)M(2j)}{256 Y_{2j}^{(2)}} + \frac{h(2j+1)M(2j+1)}{256 Y_{2j+1}^{(4)}} \right) &\leq \sum_{j=1}^{n} \frac{h(j)}{ X_{j} } \\
\sum_{j=1}^{\frac{n}{2}} \left( \frac{256 h(2j)M(2j)}{Y_{2j}^{(2)}} + \frac{256 h(2j+1)M(2j+1)}{Y_{2j+1}^{(4)}} \right) &\geq 
\sum_{j=1}^{n} \frac{h(j)}{ X_{j} } 
\end{align*}
This will allow us to approximate functionals of $X_{j}$ such as \eqref{EqSpectralGapBounders} and \eqref{BdIneqHittingMixingRelation} by substituting $Y_{2j}^{(i)}$ for $X_{j}$. \par 

\section{Cutoff Examples} \label{SecCutoffExamples}
This section proves that random birth and death chains with `IF' distribution (see \cite{LSN11} for a description of this distribution and some basic calculations) and binomial distribution can exhibit cutoff. We begin by looking at the IF distributions. These distributions $\pi = \pi_{n, \epsilon,a}$ are symmetric on $[2n-1]$. They satisfy $\pi(j) = c a^{-n + n^{\epsilon} + j - 1}$ for $1 \leq j \leq n - n^{\epsilon}$ and $\pi(j) = c$ for $n - n^{\epsilon} < j \leq n$, where $c$ is a normalizing constant, $a > 1$, $0 < \epsilon < 1$. Note that $c \asymp \frac{1}{2} n^{-\epsilon}$ for $n$ sufficiently large, in that it is uniformly within a multiplicative factor of 2, which is all that matters for the following calculations. Thus, $\pi(j) \asymp \frac{1}{2} n^{-\epsilon} a^{-n + n^{\epsilon} + j - 1}$. Let $K_{n}$ be a sequence of independent uniformly chosen birth and death chains with stationary distribution $\pi_{n}$. Then, fix a function $s: \mathbb{N} \rightarrow \mathbb{N}$. Theorem \ref{ThmIfCutoff} below characterizes the existence of cutoff for $\frac{1}{2}(I + K_{s(n)})$ in terms of the growth rate of $s$. \par 
It is worth taking a moment to understand heuristically the need for this function $s(n)$. It will turn out that, with probability about $\frac{1}{\log(n)}$, the expected hitting time of $n$ from $0$ for the kernel $\frac{1}{2} (I + K_{n})$ is comparable to the inverse of a single very small transition probability $K[i,i+1]$. When this occurs, the distribution of this hitting time is essentially given by the CDF of a geometric random variable with mean $\frac{1}{K[i,i+1]}$. Of course, this hitting time cannot possibly concentrate around its mean. For $s(n)$ growing slowly, this domination by a single small transition probability value will happen for infinitely many values of $n$, preventing cutoff. The theorem below indicates that for $s(n)$ growing quickly enough to avoid this particular obstacle, cutoff does happen. In that sense, this domination by a single small transition probability is the `only' obstacle to cutoff. It is worth pointing out that this growth rate is very rapid - even $s(n) = 2^{n}$ doesn't allow cutoff, though $s(n) = 2^{n^{2}}$ does. The following theorem generalizes the results of section 7 of \cite{DiWo10}:

\begin{thm} [Cutoff for IF Chains] \label{ThmIfCutoff} 
Fix $a>1$. For $\epsilon < \frac{1}{2}$ in the above sequence of chains, there is cutoff with probability 1 if $\sum_{n} \frac{1}{ \log(s(n)) } < \infty$. There is no cutoff with probability 1 if $\sum_{n} \frac{1}{\log(s(n))}$ diverges, or if $\epsilon \geq \frac{1}{2}$. In addition, for $\epsilon < \frac{1}{2}$ and any $a, \mu >0$,

\begin{equation*}
\lim_{n \rightarrow \infty} P\left[ \frac{1}{12 \sqrt{2} \pi} - \mu \leq \frac{\tau_{n}(\frac{1}{4})}{n \log(n)} \leq \frac{48}{1 - a^{-1}} + \mu \right] = 1
\end{equation*}
\end{thm}

First, assume $0 < \epsilon < \frac{1}{2}$ and fix $0 < \delta < 1$. Also, for $1 \leq i \leq 4$, define $Y_{t}^{(i)}$ as in the statement of Lemma \ref{BdLemCoupToStat}. Then by equation \eqref{BdEqHittingTime},
\begin{align} \label{IneqIfHitting}
E_{1}[T_{n + \delta n^{\epsilon}}] &= \sum_{v=1}^{n + \delta n^{\epsilon}} \frac{1}{\pi(v) K[v,v+1]} \sum_{q=1}^{v-1} \pi(q) \\
&= \sum_{v=1}^{n - n^{\epsilon}} \frac{1}{\pi(v) K[v,v+1]} \sum_{q=1}^{v} \pi(q) + \sum_{v = n - n^{\epsilon} + 1}^{n + \delta n^{\epsilon}} \frac{1}{\pi(v) K[v,v+1]} \sum_{q=1}^{v} \pi(q) \\
&\geq \sum_{v=1}^{n - n^{\epsilon}} \frac{1 - a^{-v-1}}{1-a^{-1}} \frac{1}{K[v,v+1]} \\
&\geq \sum_{v=1}^{\lfloor \frac{n- n^{\epsilon}}{2} \rfloor} \frac{1}{Y_{2v}^{(1)}}
\end{align}
By the symmetry of $\pi$, $E_{1}[T_{x}] \stackrel{D}{=} E_{2n-1}[T_{2n - 1 - x}]$, so it is sufficient to look only at $E_{1}[T_{x}]$ for various values of $x$. Next, we look at the spectral gap. Again, it is enough to look at the quantity $B_{-}$ described in equation \eqref{EqSpectralGapBounders}, since $B_{-} \stackrel{D}{=} B_{+}$. From equation \eqref{EqSpectralGapBounders}, we can view $B_{-}$ as a supremum over different values of $x$, and we will look at two regions for the values of $x$. \\
\textbf{Case 1: $x \leq n - n^{\epsilon}$ :} In this case,  
\begin{align*}
\left( \sum_{y=x}^{n-1} \frac{1}{\pi(y) K[y,y+1]} \right) \sum_{y=0}^{x-1} \pi(y) &= \sum_{y=x}^{n-n^{\epsilon}} \frac{1}{K[y,y+1]} a^{x-y} \frac{1 - a^{-x}}{1 - a^{-1}} \\
&+ \sum_{y= n - n^{\epsilon} +1}^{n-1} \frac{1}{K[y,y+1]} a^{-n + n^{\epsilon} + x -1} \frac{1 - a^{-x}}{1 - a^{-1}} \\
&\leq 2 \max_{x \leq y \leq n - n^{\epsilon}} \left( \frac{1}{K[y,y+1]} \right) + \frac{1}{1 - a^{-1}} \sum_{y=n - n^{\epsilon} + 1}^{n-1} \frac{1}{K[y,y+1]}
\end{align*} 
Rewriting, in case 1, all terms are at most
\begin{equation*}
\frac{2}{1-a^{-1}} \max_{1 \leq y \leq n - n^{\epsilon}} \left( \frac{1}{K[y,y+1]} \right) + \frac{1}{1 - a^{-1}} \sum_{y=n - n^{\epsilon} + 1}^{n-1} \frac{1}{K[y,y+1]}
\end{equation*}
Next, \\
\textbf{Case 2: $x > n - n^{\epsilon}$ :} In this case,
\begin{align*}
\left( \sum_{y=x}^{n-1} \frac{1}{\pi(y) K[y,y+1]} \right) \sum_{y=0}^{x-1} \pi(y) &= \left(\frac{1 - a^{-n + n^{\epsilon} + 1}}{1 - a^{-1}} \right) \sum_{y=x}^{n-1} a^{n- n^{\epsilon} - y} \frac{1}{K[y,y+1]} \\
&\leq \frac{n^{\epsilon}}{1 - a^{-1}} \sum_{y=n - n^{\epsilon}}^{n-1} \frac{1}{K[y,y+1]}
\end{align*}
Thus, putting together the two cases,
\begin{equation} \label{BdIneqSpectGapExpChain}
\frac{1 - a^{-1}}{2}B_{-} \leq \max \left(  \max_{1 \leq y \leq n - n^{\epsilon}} \left( \frac{1}{K[y,y+1]} \right), n^{\epsilon} \sum_{y=n - n^{\epsilon}}^{n-1} \frac{1}{K[y,y+1]} \right)
\end{equation}

Note that by Lemma \ref{BdLemCoupToStat} and the explicit stationary distribution for $Z_{j}$ given shortly after the proof of Theorem \ref{BdTheoremFunctClt},
\begin{equation} \label{BdIneqMaxOfManySupDiagTerms}
P \left[ \max_{x < n} \frac{1}{K[x,x+1]} > Cn \right] \leq \frac{24}{C}
\end{equation}

Next, it is necessary to look at
\begin{equation*}
P \left[\sum_{y=n-n^{\epsilon}}^{n-1} \frac{1}{K[y,y+1]} \geq C n^{1-\epsilon} \right]
\end{equation*}
Using Lemma \ref{BdLemCoupToStat}, a union bound, and changing $C$ by a universal constant independent of $n$, $a$, and $\epsilon$, this is at most
\begin{equation*}
2 P \left[\sum_{v=0}^{n^{\epsilon}} \frac{1}{Y_{2v}^{(1)}} \geq C n^{1-\epsilon} \right]
\end{equation*}
which is bounded by
\begin{equation} \label{IneqBdSumA1}
P \left[\sum_{v=0}^{n^{\epsilon}} \frac{1}{Y_{2v}^{(1)}} \geq C n^{1-\epsilon} \right] \leq P \left[\sum_{v=0}^{n^{\epsilon}} \frac{1}{Y_{2v}^{(1)}} \textbf{1}_{Y_{2v}^{(1)} \geq \frac{1}{n^{\epsilon} \log(n)}} \geq C n^{1-\epsilon} \right] + P \left[\sup_{0 \leq v \leq n^{\epsilon}} \frac{1}{Y_{2v}^{(1)}} \geq n^{\epsilon} \log(n) \right]
\end{equation}
But direct calculation shows that
\begin{equation} \label{IneqBdSumA2}
E \left[\sum_{v=0}^{n^{\epsilon}} \frac{1}{Y_{2v}^{(1)}} \textbf{1}_{Y_{2v}^{(1)} \geq \frac{1}{n^{\epsilon} \log(n)}} \right] \leq n^{\epsilon} \log(n)^{2}
\end{equation}
and so, combining Markov's inequality with inequalities \eqref{BdIneqMaxOfManySupDiagTerms}, \eqref{IneqBdSumA1} and \eqref{IneqBdSumA2},
\begin{equation} \label{IneqBadSetIf}
P \left[\sum_{v=0}^{n^{\epsilon}} \frac{1}{Y_{2v}^{(1)}} \geq C n^{1-\epsilon} \right] \leq \frac{\log(n)^{2}}{C n^{1 - 2 \epsilon}} + \frac{2}{\log(n)}
\end{equation}

In order to obtain quantitative bounds on which growth rates for $s(n)$ result in cutoff, we require the following quantitative bound on the growth rate of $\sum_{v} \frac{1}{Y_{2v}^{(1)}}$:

\begin{lemma} [Medium Deviations for Levy Sums] \label{BdLemmaStableLawRates}
\begin{equation*}
P \left[\sum_{v=0}^{\lfloor \frac{n-1}{2} \rfloor}  \frac{1}{Y_{2v}^{(1)}} < \frac{1}{\sqrt{2} \pi} n\log(n) - Cn \right] = O\left( \frac{1}{n} \right) + O \left( \frac{1}{C} \right)
\end{equation*}
\end{lemma}
As shown in section 7.1 of \cite{DiWo10}, it is possible to find a sequence of iid random variables $\zeta_{i}$ with finite exponential moments so that 
\begin{equation*}
\sum_{v=0}^{\lfloor \frac{n-1}{2} \rfloor}  \frac{1}{Y_{2v}^{(1)}} \geq \sum_{j=1}^{T} \frac{1}{v_{j}}
\end{equation*}
where $v_{j}$ is a sequence of iid random variables distributed uniformly on $[0, \frac{1}{2}]$, $T = \sup \{ i \, : \zeta_{i} \leq \lfloor \frac{n-1}{2} \rfloor \}$, and where furthermore $\zeta_{i}$ is stochastically dominated by an exponential variable with mean $\frac{\pi }{\sqrt{2}}$. By section 4 of \cite{KeKu00},
\begin{equation*}
P \left[\sum_{j=1}^{T} \frac{1}{v_{j}} \leq 2 T \log(T) - CT \right] = \Theta \left( \frac{1}{C} \right) + O \left( \frac{1}{T} \right)
\end{equation*}
Thus,
\begin{align*}
P \left[\sum_{v=0}^{\lfloor \frac{n-1}{2} \rfloor}  \frac{1}{Y_{2v}^{(1)}}  < \frac{1}{\sqrt{2 \pi}} n \log(n) - C n \right] &\leq P \left[\sum_{j=1}^{\frac{n}{\sqrt{2} \pi}} \frac{1}{v_{j}} \leq \frac{\sqrt{2} n}{4 \pi} \log(\frac{\sqrt{2} n}{4 \pi}) - Cn \right] \\
&+ P \left[\sum_{j=1}^{\frac{\sqrt{2} n}{4 \pi}} \zeta_{j} > \lfloor \frac{n-1}{2} \rfloor \right] \\
&= O \left(\frac{1}{C} \right) + O \left(\frac{1}{n} \right) + O \left(e^{-\gamma n} \right)
\end{align*}
for some $\gamma > 0$. $\square$. \par 

Applying inequalities \eqref{BdIneqSpectGapExpChain}, \eqref{BdIneqMaxOfManySupDiagTerms} and \eqref{IneqBadSetIf} to inequality \eqref{EqSpectralGapBounders}, we find:

\begin{equation} \label{BdIneqSpecGapLowerBoundA1P1}
P \left[(1 - \lambda_{n}) < \frac{1}{nA_{n}} \right] = O \left( \frac{1}{A_{n}} \right) + O \left(\frac{\log(n)}{n^{1 - 2 \epsilon}} \right) 
\end{equation}

Then, applying inequality \eqref{IneqIfHitting} and Lemma \ref{BdLemmaStableLawRates} to inequality \eqref{BdEqCutoffLocationEqHittingTime}, we find:
\begin{equation} \label{BdIneqMixingTimLowerBound}
P \left[\tau_{n} < \frac{1}{12 \sqrt{2} \pi } n \log(n) - B_{n} n \right] = O \left( \frac{1}{B_{n}} \right) + O \left( \frac{1}{n} \right)
\end{equation}

as $A_{n}$, $B_{n}$ go to $\infty$. Assume $\sum_{n} \frac{1}{\log(s(n))}$ converges. Then there exists some sequence $\omega_{n}$ so that $\sum_{n} \frac{\omega_{s(n)}}{\log(s(n))}$ converges and $\lim_{n \rightarrow \infty} \omega_{n} = \infty$. So, set $B_{n} = A_{n} = \frac{\log(n)}{\omega_{n}}$. Then, by inequalities \eqref{BdIneqSpecGapLowerBoundA1P1} and \eqref{BdIneqMixingTimLowerBound},
\begin{equation*}
P \left[\tau_{n} (1 - \lambda_{n}) < \frac{\omega_{n}}{12 \sqrt{2} \pi} - 1 \right] = O \left( \frac{\omega_{n}}{\log(n)} \right)
\end{equation*}
By Borel-Cantelli, the event that $\tau_{n} (1 - \lambda_{n}) < \frac{\omega_{n}}{12 \sqrt{2} \pi} - 1$ occurs only finitely often, and so by Theorem \ref{BdThmCutoffCrit} there is cutoff.  \par 

Next, we must show that cutoff doesn't occur in the other cases. We start with the case $\sum_{n} \frac{1}{\log(s(n))} = \infty$. The first step is a lower bound on the expected hitting times:
\begin{equation} \label{BDIneqIFHittingStuff}
E_{0}[T_{n + \delta n^{\epsilon}}] \leq \frac{12}{1 - a^{-1}} \sum_{0 \leq 2v \leq n - n^{\epsilon}}\left( \frac{1}{Y_{2v}^{(1)} + Y_{2v+1}^{(3)}} \right) + (n^{\epsilon} + 1) \sum_{ n - n^{\epsilon} \leq 2v \leq n + \delta n^{\delta}} \left( \frac{1}{Y_{2v}^{(1)} + Y_{2v+1}^{(3)}} \right)
\end{equation}
and a corresponding upper bound on $B_{-}$:
\begin{equation} \label{BDIneqIFSpecBoundB1}
B_{-} \geq \frac{1}{6(1 - a^{-1})} \max_{n - n^{\epsilon} \leq 2v \leq n} \left( \frac{1}{Y_{2v}^{(1)}} \right)
\end{equation}
Fix some constant $D$ to be determined later. For $n - n^{\epsilon} + D \log(n) \leq 2j \leq n + \delta n^{\epsilon} - D \log(n)$, let $\mathcal{A}_{j}$ be the event that $\frac{1}{Y_{2j}^{(1)}} \geq n \log(n)$.

Then, conditioning on the event $\mathcal{A}_{j}$, let $A_{t} = \frac{1}{Y_{2t+1}^{(3)}}$, $B_{t} = \frac{1}{Y_{2t}^{(1)}}$. We will couple these two chains to stationary versions, denoted by $\widehat{A_{t}}$ and $\widehat{B_{t}}$. Our goal will be to better understand the behaviour of $A_{t},B_{t}$ by showing that they agree with $\widehat{A_{t}}, \widehat{B_{t}}$ at all times more than distance $O(\log(n))$ from $2j$, with high probability. To construct our coupling, begin by choosing $A_{2j}, B_{2j}, \widehat{A_{2j}}$ and $\widehat{B_{2j}}$ independently from their respective distributions. We will then iteratively construct $(A_{2j+2 \ell}, \widehat{A_{2j + 2 \ell}})$ conditioned on $(A_{2j+2 \ell - 2}, \widehat{A_{2j + 2 \ell - 2}})$ according to the coupling used in the proof of Lemma \ref{BdLmGeneralSdMixing}. As in the proof that that lemma, this coupling has the property that if $A_{2j+2 \ell} =  \widehat{A_{2j + 2 \ell}}$, then $A_{2j+2 \ell'} =  \widehat{A_{2j + 2 \ell'}}$ for all $\ell' > \ell$. We will use the same iterative construction for the three other pairs $(A_{2j-2 \ell}, \widehat{A_{2j - 2 \ell}}), (B_{2j+2 \ell}, \widehat{B_{2j + 2 \ell}})$ and $(B_{2j-2 \ell}, \widehat{B_{2j - 2 \ell}})$. \par 

Let $\zeta_{1} = \inf \{ \ell \, : \, A_{2j + 2 \ell} = \widehat{A_{2j + 2 \ell}}, \ell \geq 0 \}$ and $\zeta_{2} = \inf \{ \ell \, : \, A_{2j - 2 \ell} = \widehat{A_{2j - 2 \ell}}, \ell \geq 0 \}$ be the coupling times of $A_{t}$ with $\widehat{A_{t}}$, and define the coupling times $\zeta_{3}, \zeta_{4}$ for the chains $B_{t}, \widehat{B_{t}}$ analogously. Finally, let $\tau_{couple} = \max_{1 \leq i \leq 4} (\zeta_{i})$ and let $\mathcal{E}_{D}$ be the event that $\tau_{couple}< D \log(n)$. Then, for $B,C$ some large constants to be determined,
\begin{align*}
P&[(1 - a^{-1})E_{0}[T_{n + \delta n^{\epsilon}}] \leq C n \log(n) \vert \mathcal{A}_{j}] \geq P \Big[\{ \sum_{2v=1}^{n - n^{\epsilon}} \left( \frac{1}{Y_{2v}^{(1)}} + \frac{1}{Y_{2v+1}^{(3)}} \right) \leq \frac{C}{100} n \log(n) \} \\
& \cap \{n^{\epsilon} \sum_{2v=j+D \log(n)}^{n + \delta n^{\epsilon}} \left( \frac{1}{Y_{2v}^{(1)}} + \frac{1}{Y_{2v+1}^{(3)}} \right) \leq \frac{C}{100} n \log(n) \} \\
&\cap \{ n^{\epsilon} \sum_{2v=j-D \log(n)}^{ j+D \log(n)} \left( \frac{1}{Y_{2v}^{(1)}} + \frac{1}{Y_{2v+1}^{(3)}} \right) \leq \frac{C}{100} n \log(n) \} \\
&\cap \{ n^{\epsilon} \sum_{2v=n - n^{\epsilon}}^{j - D \log(n)} \left( \frac{1}{Y_{2v}^{(1)}} + \frac{1}{Y_{2v+1}^{(3)}} \right) \leq \frac{C}{100} n \log(n) \} \cap \mathcal{E}_{D} \vert \mathcal{A}_{j} \Big] \\
&\geq P \left[\{ \sum_{2v=1}^{n - n^{\epsilon}} \left( \frac{1}{Y_{2v}^{(1)}} + \frac{1}{Y_{2v+1}^{(3)}} \right) \leq \frac{C}{100} n \log(n) \} \cap \mathcal{E}_{D} \vert \mathcal{A}_{j} \right] \\
&+ P \left[\{ \sum_{2v=j+D \log(n)}^{n + \delta n^{\epsilon}} \left( \frac{1}{Y_{2v}^{(1)}} + \frac{1}{Y_{2v+1}^{(3)}} \right) \leq \frac{C}{100} n \log(n) \} \cap \mathcal{E}_{D} \vert \mathcal{A}_{j} \right] \\
&+ P \left[ n^{\epsilon} \sum_{2v=j-D \log(n)}^{ j+D \log(n)} \left( \frac{1}{Y_{2v}^{(1)}} + \frac{1}{Y_{2v+1}^{(3)}} \right) \leq \frac{C}{100} n \log(n) \vert \mathcal{A}_{j} \right] \\
&+ P \left[ \{ n^{\epsilon} \sum_{2v=n - n^{\epsilon}}^{j - D \log(n)} \left( \frac{1}{Y_{2v}^{(1)}} + \frac{1}{Y_{2v+1}^{(3)}} \right) \leq \frac{C}{100} n \log(n) \} \cap \mathcal{E}_{D} \vert \mathcal{A}_{j} \right] - 3
\end{align*}

We will now analyze these four terms, beginning with the last and most complicated. To simplify notation, let $I_{j} = \{x \, : \, n - n^{\epsilon} \leq 2x \leq n + \delta n^{\epsilon}, \, \vert j - 2x \vert \geq D \log(n) \}$ and let $H_{j} = \{x \, : \, n - n^{\epsilon} \leq 2x \leq n + \delta n^{\epsilon}, \, \vert j - 2x \vert < D \log(n) \}$.
\begin{align*}
P \left[ \{ \sum_{2v=n - n^{\epsilon}}^{j - D \log(n)} \frac{1}{Y_{2v}^{(1)}} \leq \frac{C}{100} n^{1-\epsilon} \log(n) \} \cap \mathcal{E}_{D} \vert \mathcal{A}_{j} \right] &\geq 1 - P \left[ \{ \sum_{v \in I_{j} } \frac{1}{Y_{2v}^{(1)}} \geq \frac{C}{200} n^{1 - \epsilon} \log(n) \} \cap \mathcal{E}_{D} \vert \mathcal{A}_{j} \right] \\
& - P[\mathcal{E}_{D}^{c} \vert \mathcal{A}_{j}] \\
&\geq 1 - O \left( \frac{1}{C n^{1 - 2 \epsilon} \log(n)}\right) - O \left( \left( \frac{23}{27} \right)^{D \log(n)} \right) 
\end{align*}
where the two terms in the last line are bounded by Lemma \ref{BdLemmaStableLawRates} and Corollary \ref{BdCorGenMixing} respectively. The second term is bounded the same way. \par 

By inequality \eqref{IneqBadSetIf}, the third term is $O(n^{-\alpha})$ for some $\alpha > 0$ depending on $C$, for $C$ sufficiently large relative to $D$. Finally, the first term goes to 1 as $B$ goes to infinity by Theorem \ref{BdThmLimitingSumsSupDiag} and Lemma \ref{BdLmGeneralSdMixing}. So, for $B$ sufficiently large and $C$ sufficiently large relative to $B$, and for some $N_{0}$ and all $n > N_{0}$,
\begin{equation} \label{IfIneqCondHitExp}
P[E_{0}[T_{n}] \leq C n \log(n) \vert \mathcal{A}_{j}] \geq \frac{99}{100}
\end{equation}
Next, we need a lower bound on $P[\mathcal{A}_{i}]$ and an upper bound on $P[\mathcal{A}_{i} \cap \mathcal{A}_{j}]$. By Lemma \ref{LemmaBdLargenessGibbs}, $P[\mathcal{A}_{i}] = \Omega\left( \frac{1}{n \log(n)} \right)$.  By Lemma \ref{LemmaBdLargenessGibbs} and Lemma \ref{LemmaBdSmallness},

\begin{align*}
P[\mathcal{A}_{i} \cap \mathcal{A}_{j}] = O\left( \frac{1}{n^{2} \log(n)^{2}} \right)
\end{align*}

for $\vert i - j \vert > 1$. In the case $\vert i - j \vert = 1$, the same conclusion holds by Lemma \ref{LemmaBdLargenessGibbs} and the comment immediately following Lemma \ref{LemmaBdSmallness}. \par 

Let $\mathcal{B}$ be the event $\{ \max( E_{1}[T_{n + \delta n^{\epsilon}}], E_{2n-1}[T_{n - \delta n^{\epsilon}}]) < D n \log(n) \} \cap \{ \sum_{1 \leq 2v \leq n} Y_{2v}^{(1)}  + Y_{2v+1}^{(3)}  < D n \log(n) \}$. Then 
\begin{align*}
P[\cup_{n - n^{\epsilon} + D \log(n) \leq 2j \leq n + \delta n^{\epsilon} - B \log(n)} \mathcal{A}_{j} \cap \mathcal{B}] &\geq \sum_{j} P[\mathcal{A}_{j} \cap \mathcal{B}] - \sum_{i < j} P[\mathcal{A}_{i} \cap \mathcal{A}_{j} \cap \mathcal{B}] \\
&= \Omega \left( \frac{1}{\log(n)} \right) - O \left( \frac{1}{\log(n)^{2}} \right)
\end{align*}
and so in particular, 
\begin{equation*}
P[\cup_{n - n^{\epsilon} + D \log(n) \leq 2j \leq n + \delta n^{\epsilon} - D \log(n)} \mathcal{A}_{j} \cap \mathcal{B}] = O \left( \frac{1}{4 \log(n)} \right)
\end{equation*}
By Borel-Cantelli, this implies that the event $\cup_{n - n^{\epsilon} + D \log(n) \leq 2j \leq n + \delta n^{\epsilon} - D \log(n)} \mathcal{A}_{j} \cap \mathcal{B}$ happens infinitely often when $\sum_{n} \frac{1}{\log(s(n))} = \infty$. By inequalities \eqref{EqSpectralGapBounders} and \eqref{BdIneqHittingMixingRelation}, $\tau_{s(n)} ( 1- \lambda_{s(n)}) \leq D$ on this event. Thus,  $\liminf_{n \rightarrow \infty} \tau_{s(n)} (1 - \lambda_{s(n)}) < \infty$ almost surely. \par

If $\epsilon \geq \frac{1}{2}$, arguments identical to those in \cite{DiWo10} show that the spectral gap is $\Omega( n^{-2\epsilon})$ with high probability as $n$ goes to infinity. Similarly, their bound on mixing times shows that $\tau_{mix} = \Omega(n^{2 \epsilon} \log(n))$ with high probability as $n$ goes to infinity. Thus, there is no cutoff. \par 

Finally, we use equation \eqref{BdEqCutoffLocationEqHittingTime} to estimate the location of cutoff, when it occurs. The lower bound in the statement of the theorem follows from inequality \eqref{BdIneqMixingTimLowerBound}, and the upper bound follows from inequality \eqref{BDIneqIFHittingStuff}.

$\square$ \par 

Next, we will look at the binomial distribution. Define a symmetric stationary distribution $\pi_{n}(x)$ on $[n] = \{ 1, \ldots, n \}$ by $\pi_{n}(x) = 2^{-n} \binom{n}{x}$. Let $K_{n}$ be a sequence of independent uniformly chosen birth and death chains with stationary distribution $\pi_{n}$. Then, for function $s: \mathbb{N} \rightarrow \mathbb{N}$, the following theorem describes the existence of cutoff for $\frac{1}{2} (I + K_{s(n)})$ in terms of the growth rate of $s$. 

\begin{thm} [Cutoff for Binomial Distributions]
The sequence of chains described exhibits cutoff if $s(n)$ grows sufficiently rapidly, and doesn't exhibit cutoff if $s(n)$ grows so slowly that:
\begin{equation*}
\sum_{n \geq 1} \frac{1}{\sqrt{s(n)} \log(s(n)^{2})} = \infty
\end{equation*}
\end{thm}

The first step is to show that there is cutoff if $s(n)$ grows sufficiently rapidly. Begin by looking at the mixing time. By classical arguments (see e.g. chapter 18 of \cite{LPW09}), the expected hitting time of $\frac{n + \sqrt{n}}{2}$ from 0 and of $\frac{n - \sqrt{n}}{2}$ from $n$ for the underlying Metropolis chain is of order $n \log(n) + O(n)$. Thus, by Lemma \ref{BdLemCoupToStat} and then Lemma \ref{BdThmLimitingSumsSupDiag}, 
\begin{equation} \label{BinLimProbSmallHits}
\lim_{n \rightarrow \infty} P[\max(E_{0}[T_{\lfloor \frac{n + \sqrt{n}}{2} \rfloor}],E_{n}[T_{\lceil \frac{n - \sqrt{n}}{2} \rceil}]) < \frac{1}{256} n \log(n)^{2} - C_{n} n \log(n)]  \rightarrow 0
\end{equation}
as long as $C_{n}$ goes to $\infty$ along with $n$. Since $\pi([0, \frac{n + \sqrt{n}}{2}])$ is uniformly bounded away from $\frac{1}{2}$ for all $n$, $\max(E_{0}[T_{\frac{n + \sqrt{n}}{2}}], E_{n}[T_{\frac{n - \sqrt{n}}{2}}])$ is a good approximation of the hitting time for our random chain by inequality \eqref{BdIneqHittingMixingRelation}. Note that in order to get a quantitative bound on a growth rate of $s(n)$ that would imply cutoff, it would be sufficient to get a quantitative bound such as Lemma \ref{BdLemmaStableLawRates} on the rate of convergence here. \par 
Next, it is necessary to look at the spectral gap. Let $\Phi_{n}(x)$ be the probability that flipping $n$ fair coins will result in $x$ or fewer heads, and let $m$ be the median of $\pi$ (either $\frac{n}{2}$ or $\frac{n-1}{2}$ depending on parity). Also let $Z_{j} = \frac{1}{Y_{2j}^{(2)}}$ and $Q_{j} = \frac{1}{Y_{2j+1}^{(4)}}$ be defined as in Lemma \ref{BdLemCoupToStat}. Then
\begin{align} \label{BinSpecGapIneqC4}
B_{+} &= \max_{x <m} \left( \sum_{y=x}^{m} \frac{1}{\pi_{n}(y) K_{n}[y,y+1]} \sum_{y < x} \pi_{n}(y) \right) \\
&= \max_{x < m} \left( \sum_{y = x}^{m} \frac{2^{n}}{\binom{y}{n} K_{n}[y,y+1]} \Phi_{n}(x) \right)  \\
&\leq \max_{x<m} \left( \sum_{2y = x}^{m} \frac{2^{n}}{\binom{y}{n} } Z_{y} \Phi_{n}(x) + \sum_{2y = x}^{m} \frac{2^{n}}{\binom{y}{n}}  Q_{y} \Phi_{n}(x) \right) 
\end{align}
Define
\begin{equation*}
F_{n}(x) = \sum_{2y=x}^{m-1} \frac{2^{n}}{\binom{y}{n}} Z_{y} \Phi_{n}(x)
\end{equation*}
The next step is to bound $F_{n}$, and thus $B_{+}$. There are two cases.\\
\textbf{Case 1: $(n-x)^{2} > n \log(n)$} \\
In this case,
\begin{align*}
F_{n}(x) &= \sum_{2y=x}^{m-1} \frac{2^{n}}{\binom{y}{n}} Z_{y} \Phi_{n}(x) \\
&= O(1) \sum_{2y=x}^{m-1} \frac{1}{\sqrt{n}} Z_{y} 
\end{align*}
In particular, $\sup_{x \, : \, (n-x)^{2} > n \log(n)} F_{n}(x) = O(1) \sum_{2y=x}^{m-1} \frac{1}{\sqrt{n}} Z_{y}$.
Note that, by a union bound over $y$ and the explicit stationary distribution for $Z_{y}$ given immediately after Theorem \ref{BdTheoremFunctClt}, 
\begin{equation} \label{BinIneqOneExpA}
P[\sup_{1 \leq 2y \leq n-1} Z_{y} \geq n^{1.5}] = O \left( \frac{1}{\sqrt{n}} \right)
\end{equation}
By direct computation and Lemma \ref{LemmaBdSmallness},
\begin{equation} \label{BinIneqOneExpB}
E[Z_{j} \textbf{1}_{Z_{j} \leq n^{1.5}}] \leq 3 \log(n)
\end{equation}
Combining Markov's inequality with inequalities \eqref{BinIneqOneExpA} and \eqref{BinIneqOneExpB}, then,
\begin{align} \label{BinIneqC1}
P[\sup_{x \, : \, (n-x)^{2} > n \log(n)} F_{n}(x) \geq Cn] &\leq P[\sum_{2y=x}^{m-1} \frac{1}{\sqrt{n}} Z_{y} \textbf{1}_{Z_{y} \leq n^{1.5}} \geq Cn] + P[\sup_{1 \leq 2y \leq n-1} Z_{y} \geq n^{1.5}] \\
&= O\left( \frac{C \log(n)}{\sqrt{n}} \right)
\end{align}

\textbf{Case 2: $(n-x)^{2} < n \log(n)$} \\
In this case,
\begin{align*}
F_{n}(x) &= \sum_{2y=x}^{m-1} \frac{2^{n}}{\binom{y}{n}} Z_{y} \Phi_{n}(x) \\
&= O(\sqrt{n}) \sum_{2y=n - \sqrt{n \log(n)}}^{m-1} Z_{y} 
\end{align*}

And so $\sup_{x \, : \, (n-x)^{2} < n \log(n)} F_{n}(x) =  O(\sqrt{n}) \sum_{2y=n - \sqrt{n \log(n)}}^{m-1} Z_{y}$. Note that 
\begin{equation} \label{BinIneqBigSupA2}
P \left[ \sup_{n - \sqrt{n \log(n)} \leq 2y \leq n-1} Z_{y} \geq \sqrt{n \log(n)} \log(n) \right] = O \left( \frac{1}{\log(n)} \right)
\end{equation}
and that
\begin{equation} \label{BinIneqExpB2}
E[Z_{j} \textbf{1}_{Z_{j} \leq n^{1.5}}] = O(\sqrt{n \log(n)} \log(n))
\end{equation}
Combining \eqref{BinIneqBigSupA2} and \eqref{BinIneqExpB2} with Markov's inequality, 
\begin{equation} \label{BinIneqSupBound2C2}
P[\sup_{x \, : \, (n-x)^{2} < n \log(n)} F_{n}(x) \geq n \log(n)^{2 - \epsilon}] = O \left( \frac{1}{\log(n)^{0.5 - \epsilon}} \right)
\end{equation}
for any $0 < \epsilon < \frac{1}{2}$. Choose $\epsilon = \frac{1}{4}$. Combine inequalities \eqref{BinIneqC1}, \eqref{BinIneqSupBound2C2} and \eqref{BinSpecGapIneqC4} with inequality \eqref{EqSpectralGapBounders} to find that $\lim_{n \rightarrow \infty} P[\frac{1}{1 - \lambda_{n}} > n \log(n)^{1.5}] = 0$. Combining this with the bound on $\tau_{n}$ given by \eqref{BinLimProbSmallHits} and \eqref{BdIneqHittingMixingRelation} shows that cutoff occurs for $s(n)$ growing sufficiently quickly. Due to the weakness of inequality \eqref{BinLimProbSmallHits}, we have no quantitative information about the growth rate required. \par 
Next, it is necessary to show that there is no cutoff if $s(n)$ grows sufficiently slowly. Analogously to the proof of Theorem \ref{ThmIfCutoff}, let $\mathcal{A}_{j}$ be the event that $\frac{1}{K[j,j+1]} \geq n \log(n)^{2}$, for $\frac{n}{2} - \sqrt{n} \leq j \leq \frac{n}{2} -1$. By the same argument as found in the proof of Theorem \ref{ThmIfCutoff} between inequalities \eqref{BDIneqIFSpecBoundB1} and \eqref{IfIneqCondHitExp},
\begin{equation*}
P[\max(E_{0}[T_{\frac{n + \sqrt{n}}{2}}], E_{n}[T_{\frac{n - \sqrt{n}}{2}}]) \leq C n \log(n)^{2} \vert \mathcal{A}_{j}] \geq \frac{99}{100}
\end{equation*}
for some $C$ fixed and large enough and all $n > N_{0}$. By a calculation almost identical to that immediately following inequality \eqref{IfIneqCondHitExp},
\begin{align*}
P[\mathcal{A}_{j}] &= \Omega \left(\frac{1}{n \log(n)^{2}} \right) \\
P[\mathcal{A}_{j} \cap \mathcal{A}_{i}] &= O \left( \frac{1}{n^{2} \log(n)^{4}} \right) 
\end{align*}
and so
\begin{equation*}
P[ \left( \cup_{i} \mathcal{A}_{i} \right) \cap \{\max(E_{0}[T_{\frac{n + \sqrt{n}}{2}}], E_{n}[T_{\frac{n - \sqrt{n}}{2}}]) \leq C n \log(n)^{2} \}] = \Omega \left( \frac{1}{\sqrt{n} \log(n)^{2}} \right)
\end{equation*}
Since $\tau_{n} (1 - \lambda_{n}) \leq C$ on the set $\left( \cup_{i} \mathcal{A}_{i} \right) \cap \{\max(E_{0}[T_{\frac{n + \sqrt{n}}{2}}], E_{n}[T_{\frac{n - \sqrt{n}}{2}}]) \leq C n \log(n)^{2} \}$, by Borel-Cantelli, there is no cutoff with probability 1 if $\sum_{n} \frac{1}{\sqrt{s(n)} \log(s(n))^{2}}$ diverges. $\square$

\section{Non-Cutoff by Comparison to Metropolis Chains} \label{SecNonCutoff}

This section includes a theorem relating cutoff in random BD chains to cutoff in non-random chains. Let $\pi_{n}$ be a sequence of stationary distributions on $[n] = \{ 1,2,\ldots,2n \}$ with $\frac{1}{4}$, $\frac{1}{2}$ and $\frac{3}{4}$ quantiles given by $u_{n}$, $m_{n}$ and $v_{n}$ respectively. We will compare a sequence $X_{t}^{(n)}$ of random BD chains with stationary distribution $\pi_{n}$, and the `Metropolis' chain $Y_{t}^{(n)}$ with the same stationary distribution, which has transition kernel $M_{n}$ given by
\begin{align*}
M_{n}[i,i+1] &= \frac{1}{4} \min \left( 1, \frac{\pi_{n}(i+1)}{\pi_{n}(i)} \right) \\
M_{n}[i,i-1] &= \frac{1}{4} \min \left( 1, \frac{\pi_{n}(i-1)}{\pi_{n}(i)} \right) 
\end{align*}
This is the chain obtained by applying the Metropolis algorithm to the $\frac{1}{2}$-lazy simple random walk on the path. See \cite{DiSa98} for a survey of the mathematical theory of this algorithm, which has been enormously influential in the application of MCMC methods to real problems. \par 
Next, define the functional $\widehat{f_{n}}$ in $M_{1}^{\star}$ by
\begin{equation*}
\widehat{f_{n}}[x] = \max \left( \sum_{v=1}^{v_{n}-1} \frac{x(\frac{v}{n})}{\pi_{n}(v)} \sum_{q=0}^{v} \pi_{n}(q), \sum_{v=u_{n}+1}^{n} \frac{x(\frac{v}{n})}{\pi_{n}(v)} \sum_{q=v}^{n} \pi_{n}(q) \right)
\end{equation*}
Let $B_{Met,n}$ be the maximum value of the quantities $B_{+}$ and $B_{-}$ defined by equation \eqref{EqSpectralGapBounders} for the Metropolis chain. Then let $x_{n}$ be a sequence satisfying one of the following inequalities for each $n$ and some fixed $\alpha > 0$:
\begin{align*}
\sum_{y=x_{n}}^{m_{n}-1} \frac{1}{M[y,y+1] \pi_{n}(y)} \sum_{y < x_{n}} \pi_{n}(y) &> \alpha B_{Met,n} \\
\sum_{y=m_{n}+1}^{x_{n}} \frac{1}{M[y,y+1]\pi_{n}(y-1)} \sum_{y > x_{n}} \pi_{n}(y) &> \alpha B_{Met,n} \\
\end{align*}
Assume without loss of generality that it is the first of these two inequalities that is being satisfied from now on. Then define the functional $\widehat{g_{n}}$ in $M_{1}^{\star}$ by
\begin{equation*}
\widehat{g_{n}}(x) = \sum_{v=x_{n}}^{m_{n}-1} \frac{x(\frac{v}{n})}{\pi_{n}(v)} \sum_{v<x_{n}} \pi_{n}(v)
\end{equation*}
Assume without loss of generality that $E_{0}[T_{v_{n}}] \geq E_{n}[T_{u_{n}}]$ for the Metropolis chain, and call this larger quantity $\tau_{Met,n}$. By inequality \eqref{BdIneqHittingMixingRelation}, this is close to the actual mixing time of the Metropolis chain. Let $(1 - \lambda_{Met,n})$ be the spectral gap of the Metropolis chain. Define normalizing constants
\begin{align*}
\beta_{1,n}^{-1} &= \sum_{v=0}^{v_{n} - 1} \frac{1}{\pi_{n}(v)} \sum_{q=0}^{v} \pi_{n}(q) \\
\beta_{2,n}^{-1} &= \sum_{y= x_{n}}^{m_{n}-1} \frac{1}{\pi_{n}(y)} \sum_{y < x_{n}} \pi_{n}(y)
\end{align*}
and let the normalized functionals be given by $f_{n} = \beta_{1,n} \widehat{f_{n}}$ and $g_{n} = \beta_{2,n} \widehat{g_{n}}$. Then

\begin{thm} [Weak Comparison to Metropolis] \label{BdThmWeakMetropolisComparison}
Assume $f_{n} \rightarrow f$ and $g_{n} \rightarrow g$ in the $M_{1}^{\star}$ topology, for some continuous functionals $f$ and $g$ and suitable normalizing sequences $\beta_{1,n}$ and $\beta_{2,n}$. Further assume that $x_{n}$, $u_{n}$, $n-v_{n}$, and $(m_{n}-x_{n})$ all go to infinity as $n$ does. Then if $Y_{t}$ does not exhibit cutoff, $X_{t}$ does not either with probability 1.
\end{thm}

The proof follows the same basic outline as found in section 7 of \cite{DiWo10}. First, it is necessary to bound the expected hitting time from above.  

\begin{lemma}[Hitting Time for Metropolized Chains] \label{LemmaLastHittingBound}
Under the conditions of Theorem \ref{BdThmWeakMetropolisComparison}, for any sequence $C_{n} \rightarrow \infty$,
\begin{equation*}
\lim_{n \rightarrow \infty} P[E_{0}[T_{v_{n}}] > 512 \tau_{Met,n} ( C_{n} + \log(n))] = 0
\end{equation*}
\end{lemma}
Defining $Y_{2t}^{(i)}$ as in Lemma \ref{BdLemCoupToStat},
\begin{align*}
E_{0}&[T_{v_{n}}] \leq 256 \sum_{2i=0}^{v_{n}} \frac{1}{\pi_{n}(2i) Y_{2i}^{(1)}} \sum_{j=0}^{i} \pi_{n}(j) + 256 \sum_{2i+1=0}^{v_{n}} \frac{1}{\pi_{n}(2i) Y_{2i+1}^{(3)}} \sum_{j=0}^{i} \pi_{n}(j)\\
&\geq 256 \left( \sum_{2i=0}^{v_{n}} \frac{\log(n)}{\pi_{n}(2i) } \sum_{j=0}^{i} \pi_{n}(j) + \sum_{2i=0}^{v_{n}} \frac{1}{\pi_{n}(2i)} \left( \frac{1}{ Y_{2i}^{(1)}} - E \left[ \frac{1}{ Y_{2i}^{(1)}} \textbf{1}_{\frac{1}{ Y_{2i}^{(1)}} \leq 2n} \right] \right) \sum_{j=0}^{i} \pi_{n}(j)  \right) \\
&+ 256 \left( \sum_{2i+1=0}^{v_{n}} \frac{\log(n)}{\pi_{n}(2i+1) } \sum_{j=0}^{i} \pi_{n}(j) + \sum_{2i+1=0}^{v_{n}} \frac{1}{\pi_{n}(2i+1)} \left( \frac{1}{ Y_{2i+1}^{(3)}} - E \left[ \frac{1}{ Y_{2i+1}^{(3)}} \textbf{1}_{\frac{1}{ Y_{2i+1}^{(3)}} \leq 2n} \right] \right) \sum_{j=0}^{i} \pi_{n}(j)  \right)
\end{align*}
The first term is at most $256 \tau_{Met,n}$, while by Theorem \ref{BdThmLimitingSumsSupDiag} and the assumptions, the second term multiplied by $\beta_{1,n}$ converges in distribution to $f(V)$, with $V$ a Levy variable.  The same bounds apply to the third and fourth terms respectively. $\square$ \par 
Next, it is necessary to examine the spectral gap. Let $(1 - \lambda_{n})$ be the gap of the random chain. Then, following a similar calculation,

\begin{align*}
\frac{1}{1 - \lambda_{n}} &\geq \frac{1}{4} \sum_{y=x_{n}}^{m_{n}-1} \frac{1}{K[y,y+1] \pi(y)} \sum_{y < x_{n}} \pi_{n}(y) \\
&\geq \frac{\alpha}{48} \frac{1}{1 - \lambda_{Met,n}}(\log(n) + \psi_{n})
\end{align*}
where, again, $\psi_{n}$ converges in distribution to a Levy variable. Thus, if $\lim_{n \rightarrow \infty} C_{n} = \infty$, 
\begin{equation} \label{IneqLastSpecBoundReally}
\lim_{n \rightarrow \infty} P[\frac{1}{1 - \lambda_{n}} < \frac{\alpha}{48 (1 - \lambda_{Met,n})}(\log(n) - C_{n})] = 0
\end{equation}

Putting together Lemma \ref{LemmaLastHittingBound} and inequality \eqref{IneqLastSpecBoundReally}, and setting $C_{n} = \log(n)^{0.1}$, gives
\begin{equation} \label{IneqReallySimple}
\lim_{n \rightarrow \infty} P[\tau_{n} (1 - \lambda_{n}) > \frac{24576}{\alpha} \tau_{Met,n} (1 - \lambda_{Met,n})] = 0
\end{equation}
Since there exists some constant $C$ so that $\tau_{Met,n} (1 - \lambda_{Met,n}) < C$ infinitely often, by the Borel-Cantelli Lemma and \ref{IneqReallySimple}, the random chain satisfies $\tau_{n} (1 - \lambda_{n}) < \frac{24576}{\alpha} C$ infinitely often. $\square$ \par

Because of the slightly exotic metrics used in its definition, it might not be clear that Theorem \ref{BdThmWeakMetropolisComparison} can actually be used. To illustrate its applications, the following is a short second proof of the cutoff result in \cite{DiWo10} using this theorem. Set $\pi_{n}(i) = \frac{1}{n}$ for all $1 \leq i \leq n$. Then set $x_{n} = \frac{n}{4}$. It is clear that
\begin{equation*}
\sum_{y = \frac{n}{2}}^{\frac{n}{2}} \frac{4}{\frac{1}{n}} \sum_{y=0}^{\frac{n}{4}} \frac{1}{n} = \frac{n^{2}}{4}
\end{equation*} 
is within a constant factor of the inverse spectral gap of the Metropolis chain, so the choice of $x_{n}$ is allowed. We also note that $u_{n} = \frac{n}{4} + O(1)$, $m_{n} = \frac{n}{2} + O(1)$, and $v_{n} = \frac{3n}{4} + O(1)$. Next, note that
\begin{align*}
f_{n}(\lambda) &= \sum_{v=0}^{\frac{3n}{4}} \frac{v}{n} \lambda( \frac{v}{n}) + O \left( \frac{ \vert \vert \lambda \vert \vert_{\infty}}{n} \right)\\
g_{n}(\lambda) &= \sum_{v= \frac{n}{4}}^{\frac{n}{2}} \frac{1}{n} \lambda(\frac{v}{n}) + O \left( \frac{ \vert \vert \lambda \vert \vert_{\infty}}{n} \right)
\end{align*}
The last step is to show that
\begin{align*}
f_{n}(\lambda) &\rightarrow \int_{0}^{\frac{3}{4}} v \lambda(v) dv\\
g_{n}(\lambda) &\rightarrow \int_{\frac{1}{4}}^{\frac{1}{2}} \lambda(v) dv
\end{align*}
in the $M_{1}^{\star}$ topology. It is clear that this convergence holds when applied to all continuous functions $\lambda \in D[0,1]$, and that these functions are dense in $D[0,1]$ under the $M_{1}$ topology, so it is sufficient to show that the limiting functions are continuous in $M_{1}^{\star}$. It is easiest to look at the limiting functional $g$, though the proof for $f$ is essentially identical. Let $\lambda_{n} \rightarrow \lambda$ in the $M_{1}$ topology. By Theorem 2.4.1 of \cite{Skor56} and an application of the triangle inequality, for all $\delta > 0$, there exists some $N$ such that for all $n > N$,
\begin{equation*}
\inf_{v - \delta \leq u \leq v + \delta} \lambda(u) - \delta \leq \lambda_{n}(v) \leq \sup_{v - \delta \leq u \leq v + \delta} \lambda(u) + \delta
\end{equation*}
Lets look at only the upper bound, since the lower is identical.
\begin{equation*}
\int_{0}^{\frac{3}{4}} (\lambda_{n}(v) - \lambda(v)) dv \leq \int_{0}^{\frac{3}{4}} \delta dv + \int_{0}^{\frac{1}{2}} (\sup_{v - \delta \leq u \leq v + \delta} \lambda(u) - \lambda(v)) dv
\end{equation*}
The first term clearly goes to 0 as $\delta$ goes to 0, and the second goes to 0 by the dominated convergence theorem. Thus, integration is continuous, and the theorem can be applied in this case. \par 

We conjecture that divergence of $\sum_{n} \frac{1}{\tau_{s(n)} (1 - \lambda_{s(n)})}$ for the Metropolis chain also implies lack of cutoff for the random chain, but have not been able to prove it. The simple conjecture just comes from assuming that unusually large values of $\frac{1}{1 - \lambda_{n}}$ don't tend to occur especially often with unusually large values of $\tau_{n}$. In particular, this conjecture would hold if the sums defining the expected hitting times and spectral gap were all independent. Note also that this conjecture is based only on the spectral gap of the original chain, and as shown by Theorem \ref{ThmIfCutoff}, it cannot be sharp.

\bibliographystyle{plain}
\bibliography{BdCutoffBib}

\end{document}